\documentclass[12pt]{amsart}
\usepackage{amsmath}
\usepackage{amssymb}
\usepackage{amscd}
\newtheorem{theorem}{Theorem}[section]
\newtheorem{lemma}[theorem]{Lemma}

\theoremstyle{definition}
\newtheorem{definition}[theorem]{Definition}
\newtheorem{definitions}[theorem]{Definitions}
\newtheorem{example}[theorem]{Example}
\newtheorem{examples}[theorem]{Examples}
\newtheorem{definitions and remarks}[theorem]{Definitions and Remarks}
\newtheorem{question}[theorem]{Question}

\theoremstyle{remark}
\newtheorem{remark}[theorem]{Remark}
\newtheorem{remarks}[theorem]{Remarks}

\numberwithin{equation}{section}

\newcommand{\exc}{\mathrm{exc}}
\newcommand{\inv}{\mathrm{inv}}

\newcommand{\Sing}{\mathrm{Sing}\,}

\newcommand{\Supp}{\mathrm{Supp}\,}
\newcommand{\cosupp}{\mathrm{cosupp}\,}

\newcommand{\ord}{\mathrm{ord}}

\newcommand{\car}{\mathrm{char}\,}

\newcommand{\lcm}{\mathrm{lcm}}

\newcommand{\Bl}{\mathrm{Bl}}

\newcommand{\al}{{\alpha}}
\newcommand{\be}{{\beta}}

\newcommand{\ep}{{\epsilon}}

\newcommand{\la}{{\lambda}}

\newcommand{\p}{{\partial}}
\newcommand{\s}{{\sigma}}

\newcommand{\vp}{{\varphi}}
\newcommand{\io}{{\iota}}

\newcommand{\IN}{{\mathbb N}}

\newcommand{\IA}{{\mathbb A}}

\newcommand{\IC}{{\mathbb C}}

\newcommand{\IZ}{{\mathbb Z}}

\newcommand{\cC}{{\mathcal C}}
\newcommand{\cD}{{\mathcal D}}

\newcommand{\cG}{{\mathcal G}}

\newcommand{\cI}{{\mathcal I}}
\newcommand{\cJ}{{\mathcal J}}
\newcommand{\cM}{{\mathcal M}}

\newcommand{\cO}{{\mathcal O}}

\newcommand{\cR}{{\mathcal R}}
\newcommand{\cS}{{\mathcal S}}

\newcommand{\uk}{\underline{k}}

\newcommand{\tx}{{\tilde x}}
\newcommand{\ty}{{\tilde y}}
\newcommand{\ts}{{\tilde \sigma}}

\newcommand{\ucG}{\underline{\cG}}

\newcommand{\ucI}{\underline{\cI}}
\newcommand{\ucJ}{\underline{\cJ}}
\newcommand{\ucC}{\underline{\cC}}
\newcommand{\ucD}{\underline{\cD}}

\newcommand{\ucM}{\underline{\cM}}

\newcommand{\ucR}{\underline{\cR}}

\begin{document}

\title[Resolution except for minimal singularities I]
{Resolution except for minimal singularities I}

\author{Edward Bierstone}
\address{The Fields Institute, 222 College Street, Toronto, Ontario, Canada
M5T 3J1, and University of Toronto, Department of Mathematics, 40 St. George Street,
Toronto, Ontario, Canada M5S 2E4}
\email{bierston@fields.utoronto.ca}
\thanks{Research supported in part by the following grants: Bierstone: NSERC
OGP0009070 and MRS342058, Milman: NSERC OGP0008949.}

\author{Pierre D. Milman}
\address{University of Toronto, Department of Mathematics, 40 St. George Street,
Toronto, Ontario, Canada M5S 2E4}
\email{milman@math.toronto.edu}



\begin{abstract}
The philosophy of this article 
is that the desingularization invariant together 
with natural geometric information can be used to compute local normal forms of 
singularities.  The idea is used in two related problems: 
(1) We give a proof of resolution 
of singularities of a variety or a divisor, except for simple normal crossings  (i.e., which 
avoids blowing up simple normal crossings, and ends up with a variety or a divisor 
having only simple normal crossings singularities). (2) For more general normal 
crossings (in a local analytic or formal sense), such a result does not hold. We find 
the smallest class of singularities (in low dimension or low codimension) with which 
we necessarily end up if we avoid blowing up normal crossings singularities. Several of 
the questions studied were raised by Koll\'ar.
\end{abstract}

\maketitle
\setcounter{tocdepth}{1}
\tableofcontents

\section{Introduction}\label{sec:intro} 
The philosophy developed in this article and in the sequel \cite{BLM} 
is that the desingularization invariant of
\cite{BMinv} together with natural geometric information can be used to compute
local normal forms of singularities, at least when the constant locus of the invariant
has low codimension.  The idea is used in two related problems: 
(1) We give a proof 
of resolution of singularities of a variety or a divisor, except for simple normal crossings
(i.e., which avoids blowing up simple normal crossings singularities, and ends up
with a variety or a divisor having only simple normal crossings singularities). (2) For 
more general normal crossings (in a local analytic or formal sense), such a result does
not hold. We find the smallest class of singularities (in low dimension or low 
codimension) with which we necessarily end up if we avoid blowing up normal 
crossings singularities. Several of the questions studied were 
raised by J\'anos 
Koll\'ar \cite{Kolog}. 
We have included a \emph{Crash course on the desingularization invariant} as an 
Appendix, in order to make the article as self-contained as possible.

The preceding problems are interesting because normal crossings or more general
``mild singularities'' have to be admitted in natural geometric situations.

\begin{example}\label{ex:elliptic}
Consider the family of projective curves $X_\la$, 
$$
z^3 + y^3 + x^3 - 3\la xyz = 0.
$$
The curve $X_\la$ is smooth if $\la^3 \neq 1$. When $\la = 1$, for example,
the equation splits as
$$
(z+y+x)(z+ \ep y + \ep^2 x)(z + \ep^2 y + \ep x) = 0, 
$$
where $\ep$ denotes the cube root of unity $\ep = e^{2\pi i/3}$; in particular
$X_1$ has normal crossings singularities. We cannot simultaneously 
resolve the singularities
of a family of curves without allowing special fibres that have normal crossings
singularities. (Here, for instance, because the generic and special
fibres have different genera.)
\end{example}

As another example, resolution of singularities of an ideal or a divisor (``log-resolution'' 
of singularities) leads to a divisor with normal crossings.
In the same way, when we resolve the singularities of a singular algebraic (or
analytic) variety, its total transform (or inverse image, with respect to any local
embedding of the variety in a smooth space) necessarily has normal crossings
singularities. From the point of view of these examples, it is reasonable to consider 
normal crossings singularities acceptable from the start (in any case, they can be
eliminated by normalization), and to ask whether we can
resolve singularities except for normal crossings. In particular, we can ask:

\begin{question}\label{qu:nc}
Given an algebraic variety $X$, can we find a proper birational morphism
$\s: X' \to X$ such that
\begin{enumerate}
\item $X'$ has only normal crossings singularities;
\item $\s$ is an isomorphism over the locus of points of $X$ having only normal
crossings singularities?
\end{enumerate}
\end{question}

An \emph{algebraic variety} means a scheme of finite type over a field $\uk$.
Throughout this article, $\car \uk = 0$.

The question above is ambiguous. Roughly speaking, we say that $X$ has normal
crossings at a point $a$ if, locally at $a$, every irreducible component is smooth and
all intersections are transverse; in other words, locally, $X$ can be embedded in a smooth
variety $Z$ with local coordinates $(x_1,\ldots, x_n)$ at $a$ in which $X$ is defined by
a monomial equation 
\begin{equation}\label{eq:mon}
x_1^{\al_1}\cdots x_n^{\al_n} = 0
\end{equation}
(where the $\al_i$ are nonnegative integers).
The ambiguity is in the meaning of ``locally'' or ``local coordinates''.

\begin{definitions}\label{def:nc}
Let $X$ denote an algebraic variety over $\uk$. We say that $X$ has \emph{simple
normal crossings} (\emph{snc}) at a point $a$ if there is an embedding of an open 
neighbourhood of $a$ in a smooth variety $Z$ and a regular system of parameters
$(x_1,\ldots, x_n)$ for $Z$ at $a$, with respect to which $X$ is defined by an equation
(\ref{eq:mon}).

We say that $X$ has \emph{normal crossings} (\emph{nc}) at $a$ if the same condition
is satisfied, except that $(x_1,\ldots, x_n)$ is a local \'etale coodinate system.

We will say that $X$ has \emph{normal crossings} (or \emph{simple normal crossings})
\emph{of order} $k$ at $a$ if precisely $k$ exponents $\al_i$ are nonzero in (\ref{eq:mon}).
\end{definitions}

A variety $X$ has normal crossings at $a$ if and only if it can be defined at $a$ by a monomial
equation with respect to formal coordinates, after a finite extension of the ground field $\uk$.
In the case of simple normal crossings (with reference to the definition above), each
irreducible component of $X$ containing $a$ is given locally by $x_i = 0$, for some $i$.
The definitions \ref{def:nc} have obvious analogues for an embedded variety $X$ or for
a divisor on a smooth variety.

\begin{examples}\label{ex:nc}
The plane curve $y^2 = x^2 + x^3$ has normal crossings but not simple normal
crossings at the origin. The curve $y^2 + x^2 = 0$ is nc, but is snc if and only if
$\sqrt{-1} \in \uk$. An embedded hypersurface defined at a point by an equation
$y^2 + ux^2 = 0$, where $x,y$ are regular coordinates and $u$ is a unit in the local
ring, is nc at $a$, but snc if and only if $u$ is a square.
\end{examples}

The answer to Question \ref{qu:nc} is ``yes'' for snc (Theorem \ref{thm:nc} following),
but ``no'' for nc in general (Example \ref{ex:pp}).

\begin{theorem}\label{thm:nc}
Let $X$ denote a reduced variety over $\uk$. Let $X^{\mathrm{snc}}$ denote the
simple normal crossings locus of $X$. Then there is a morphism $\s: X'\to X$ which
is a composite of finitely many \emph{admissible} blowings-up, such that
\begin{enumerate}
\item $X' = (X')^{\mathrm{snc}}$;
\item $\s$ is an isomorphism over $X^{\mathrm{snc}}$.
\end{enumerate}
\end{theorem}

An \emph{admissible} blowing-up means a blowing-up $\s$ with centre $C$ which is
smooth and has only simple normal crossings with respect to the exceptional
divisor. The latter condition means that, with respect to a suitable local embedding
of $X$ in a smooth variety $Z$ and the induced blowing-up sequence of $Z$,
there are regular coordinates $(x_1,\ldots, x_n)$ at any point of $C$, in which 
$C$ is a coordinate subspace and each component of the exceptional divisor is a
coordinate hyperplane $(x_i = 0)$, for some $i$.

Versions of Theorem \ref{thm:nc} were first proved by Szab\'o \cite{Sz} and by the
authors \cite[\S12]{BMinv}. We give a proof in Section \ref{sec:nc} below that we
sketched in a letter to Michael Temkin (2007); see \cite[Thm.\,2.2.11]{Tem1}.
The theorem can be strengthened in various ways. (See Section \ref{sec:nc}.)
For example, instead of using the snc locus, we can use the locus of points
having only simple normal crossings singularities of order up to $r$ (\emph{snc$\leq$$r$}),
for given $r$ (Remark \ref{rem:nc}). Moreover, $\s$ can be realized as a composite
of smooth blowings-up
\begin{equation}\label{eq:blupX}
X = X_0 \stackrel{\s_1}{\longleftarrow} X_1 \longleftarrow \cdots
\stackrel{\s_{t}}{\longleftarrow} X_{t} = X'\,.
\end{equation}
where we avoid blowing up snc singularities at every step; i.e.,
each centre of blowing up is disjoint from the snc locus 
of the corresponding total transform of $X$ (with respect to a local embedding
of $X$ in a smooth variety); 
see \cite{BDV}. One can also resolve singularities of pairs, preserving ``semi-simple 
normal crossings'' \cite{BV} (see \cite[Prob.\,19]{Ko}). In Theorem \ref{thm:nc}, we can add
the following conditition, considered by Koll\'ar \cite{Ko}:
\begin{enumerate}
\item[(3)] \emph{The morphism $\s$ maps the singular set $\Sing X'$ birationally onto
the closure of $\Sing X^{\mathrm{snc}}$.}
\end{enumerate}

\begin{remark}\label{rem:funct}
Because of the way the invariant
is used, Theorem \ref{thm:nc} and the other desingularization results here are functorial.
For example, Theorem \ref{thm:nc} is functorial with respect to local isomorphisms
(or, more generally, with respect to \'etale or smooth morphisms that preserve the
number of irreducible components at every point). We will not always explicitly mention
functoriality in the statements of the theorems. (See also Remarks \ref{rem:functclean}
and \ref{rem:onfunctclean}.)
Koll\'ar gives another (non-functorial)
proof of Theorem \ref{thm:nc} in \cite{Ko}. All the desingularization results here also
have analytic versions (where the analogue of a morphism that is a finite composite
of blowings-up is a morphism which can be realized by a finite blowing-up sequence
over any relatively compact open set).
\end{remark}

Our proof of Theorem \ref{thm:nc} automatically provides the additional condition
(3) above. Given a morphism $\s: X' \to X$ satisfying the conditions of Theorem 
\ref{thm:nc}, we can also get (3) by successively blowing up every component of 
$\Sing X'$ that does not map birationally onto a component of the closure of 
$\Sing X^{\mathrm{snc}}$ (although here we would have to be a little careful to
preserve the condition of functoriality).

\begin{example}\label{ex:pp}
The \emph{pinch point} (\emph{pp}) or \emph{Whitney umbrella} $X \subset \IA^3$ is
defined by $z^2 + x y^2 = 0$. $X$ has only nc2 singularities outside the pinch point $0$.
There is no birational morphism $\s: X' \to X$ satisfying the analogues of (1), (2) of
Theorem \ref{thm:nc} with nc instead of snc, according to the following argument of
Koll\'ar \cite[\P8]{Ko} (see also Fujino \cite[Cor.\,3.6.10]{Fuji}): 
At any nonzero point of the $x$-axis, $X$ has two local analytic
branches (over $\IC$, say). As we go around the origin, the two branches are interchanged.
This continues to hold after any birational map that is an isomorphism over the generic point
of the $x$-axis, so we cannot eliminate the pinch point without blowing up the $x$-axis.
\end{example}

The desingularization invariant of \cite{BMinv} seems particularly well-suited to studying 
the questions above, as already
evidenced by our proof of Theorem \ref{thm:nc}. One of our goals is to demonstrate
that the invariant is a useful tool for making local computations in
algebraic geometry and singularity theory.
\smallskip

The authors are grateful to Franklin Vera Pacheco for many important comments
on the results in this article.

\subsection{Minimal singularities}\label{subsec:min}
Because the nc-analogue of Theorem \ref{thm:nc} fails, it is interesting to ask the
following (a variant of a question of Koll\'ar).

\begin{question}\label{qu:min}
Can we find the smallest class of singularities $\cS$ with the following properties:
\begin{enumerate}
\item $\cS$ includes all nc singularities;
\item given a reduced variety $X$, there exists a proper (birational) morphism
$\s: X' \to X$ such that
\begin{enumerate}
\item $X' = (X')^\cS$,
\item $\s$ is an isomorphism over $X^{\mathrm{nc}}$\,?
\end{enumerate}
\end{enumerate}
\end{question}

($X^\cS$ denotes the locus of points of $X$ having only singularities in $\cS$,
so that $X^\cS$ includes all smooth points.) We can also ask: Do we get the same 
class of singularities $\cS$ if, in condition (2), we require a morphism $\s$ 
which is a finite composite of admissible blowings-up?

\begin{remarks}\label{rem:min}
(1) We are interested in writing \emph{normal forms} for the singularities in $\cS$; i.e.,
local models for their equivalence classes with respect to 
\'etale coordinate changes (or, equivalently, with respect to completion and finite field
extension).
\smallskip

(2) Normal crossings singularities are singularities of hypersurfaces. We say that
$X$ is a \emph{hypersurface} if, locally, $X$ can be defined by a principal ideal on
a smooth variety. (We say that $X$ is an \emph{embedded hypersurface} if $X
\hookrightarrow Z$, where $Z$ is smooth and $X$ is defined by a principal ideal on $Z$.)
Question \ref{qu:min} can be reduced to the case of a hypersurface using the strong
desingularization algorithm of \cite{BMinv, BMfunct}. The algorithm involves blowing up
with smooth centres in the maximum strata of the Hilbert-Samuel function. The latter
determines the local embedding dimension, so the algorithm first eliminates points of
embedding codimension $> 1$ without modifying nc points.
\end{remarks}

We therefore reduce Question \ref{qu:min} to the case that, locally, $X \hookrightarrow Z$
is an embedded hypersurface, so we want to give normal forms for the singularities in $\cS$
in terms of \'etale local coordinates $(x_1,\ldots,x_n)$ for $Z$. The table in Definition \ref{def:min} following gives normal forms for $\cS$, for embedding dimension $n \leq 4$, and therefore answers Question \ref{qu:min} for varieties $X$ of dimension $\leq 3$
(at least with respect to morphisms that are composites of admissible blowings-up, but
see also Remark \ref{rem:minbirat} below).

\begin{definition}\label{def:min}
Let $\cS$ denote the following class of singularities in $n$ variables, for $n \leq 4$:
\medskip

\renewcommand{\arraystretch}{1.1}
\begin{tabular}{l r l}
$n = 2$ &                   $xy = 0$ & double normal crossings \emph{nc2}\\
&&\\[-.2cm]
$n = 3$ &                   $xy = 0$ & \emph{nc2} \\
             &                 $xyz = 0$ & triple normal crossings \emph{nc3}\\
             &            $z^2 + xy^2 = 0$ & pinch point \emph{pp}\\
&&\\[-.2cm]
$n = 4$ &                  $xy = 0$ & \emph{nc2} \\
             &                $xyz = 0$ & \emph{nc3}\\
             &              $xyzw = 0$ & \emph{nc4}\\
             &            $z^2 + xy^2 = 0$ & \emph{pp}\\ 
             &     $z^2 + (y + 2x^2)(y - x^2)^2 = 0$ & degenerate pinch point \emph{dpp}\\
             &         $x(z^2 + wy^2) = 0$ & product \emph{prod}\\
             &     $z^3 + wy^3 + w^2x^3 -3wxyz = 0$ & cyclic point \emph{cp3}\\
\end{tabular}
\end{definition}
\smallskip

\begin{theorem}\label{thm:min}
Let $X$ denote a reduced variety of pure dimension $n-1$, where $n = 2, 3,\text{ or }4$. Then there is a morphism $\s: X' \to X$ given by a finite sequence of
admissible blowings-up
\begin{equation}\label{eq:blupX.4}
X = X_0 \stackrel{\s_1}{\longleftarrow} X_1 \longleftarrow \cdots
\stackrel{\s_{t}}{\longleftarrow} X_{t} = X'\,,
\end{equation}
such that
\begin{enumerate}
\item[(a)] $X' = (X')^{\cS}$,
\item[(b)] $\s$ is an isomorphism over $X^{\mathrm{nc}}$.
\end{enumerate}
Moreover, the morphism $\s = \s_X$ (or the entire blowing-up sequence
\eqref{eq:blupX.4}) can be realized in a way that is functorial with respect
to \'etale morphisms.
\end{theorem}

The list of singularities in the case $n=3$ above was proposed by Koll\'ar \cite{Kolog}.
Theorem \ref{thm:min} for $n \leq 3$ will be proved in this article 
(see \S\ref{subsec:inv.1} below).
The case $n=4$ has been proved in collaboration with Pierre Lairez and is the subject
of the sequel \cite{BLM}. In each case, $\cS$ is the smallest class of singularities
satisfying the theorem; see Remark \ref{rem:minbirat}.

We do not have full lists of candidates for the singularities in $\cS$, for $n \geq 5$, 
though we can make a few remarks: For any $n$, $\cS$ will include a cyclic
point singularity cp($n-1$) which is an irreducible limit of nc($n-1$) singularities
along a smooth curve (see \cite{BLM}). For example, cp3 above is the singularity at the origin
of an irreducible hypersurface having nc3 singularities along the nonnegative $w$-axis.
The cyclic singularity cp$k$ of order $k$ is related to the action of the cyclic
group $\IZ_k$ of order $k$ on $\IC^k$ by permutation of coordinates. Cyclic
singularities are higher-dimensional versions of the pinch point: pp $=$ cp2.

For any $n$, $\cS$ will include singularities that occur as limits of nc($n-1$), according to
the way that the limit factors (i.e., according to an associated monodromy group); the
reducible limits will be various products of cp$k$, $k < n-1$ (where, by convention,
cp1 means a smooth point $x=0$), generalizing prod in theorem \ref{thm:min}.

Any singularity that occurs in an arbitrarily small neighbourhood of a singularity 
in $\cS$ necessarily
also belongs to $\cS$. Degenerate pinch points occur along the nonnegative $x$-axis of
cp3 (see \cite[\S2.2]{BLM}). 
The name comes from the fact that a pinch point can be rewritten as
$z^2 + (y + 2x)(y - x)^2 = 0$ after a coordinate change (see also Lemma \ref{lem:pp}).

An optimistic reader can ask whether, in any dimension $n$, $\cS$ comprises nc singularities,
products of cp$k$ singularities ($k \leq n-1$), and singularities that occur in arbitrarily small 
neighbourhoods of the latter.

There are many interesting variations of Question \ref{qu:min}. For example:

\begin{question}\label{qu:minstrong}
Can we find the smallest class of singularities $\cS'$ with the following properties:
\begin{enumerate}
\item $\cS'$ includes all nc singularities;
\item given a reduced variety $X$, there exists a proper (birational) morphism
$\s: X' \to X$ such that
\begin{enumerate}
\item $X' = (X')^{\cS'}$,
\item $\s$ is an isomorphism over $X^{\cS'}$\,?
\end{enumerate}
\end{enumerate}
\end{question}

Again we can ask: 
Do we get the same class of
singularities if, in condition (2), we require a morphism $\s$ which is a finite composite
of admissible blowings-up? For either Question \ref{qu:min} or \ref{qu:minstrong},
we can also ask: If $X$ is an embedded hypersurface, can we find the smallest 
class of corresponding singularities of the total transform (inverse image) of $X$? Are the
preceding questions well-formulated --- in each case, is there 
a (unique) smallest class of singularities satisfying the conditions stated?

Clearly, $\cS \subset \cS'$, for either version of Questions 
\ref{qu:min} and \ref{qu:minstrong}. In fact, the classes coincide for $n \leq 3$, 
but not in general.

\begin{definition}\label{def:minstrong}
If $n \leq 3$, let $\cS' := \cS$, where the latter is given by Definition \ref{def:min}.
For $n=4$, let $\cS'$ by given by the singularities in $\cS$
together with the following:
\begin{equation}\label{eq:exc}
z^2 + y(wy + x^2)^2 = 0 \quad \text{exceptional singularity \emph{exc}}
\end{equation}
\end{definition}
\smallskip

\begin{theorem}\label{thm:minstrong}
Let $X$ denote a reduced variety of pure dimension $n-1$, where $n = 2, 3,\text{ or }4$. Then there is a morphism $\s: X' \to X$ given by a finite sequence of
admissible blowings-up
\begin{equation}\label{eq:blupX.4}
X = X_0 \stackrel{\s_1}{\longleftarrow} X_1 \longleftarrow \cdots
\stackrel{\s_{t}}{\longleftarrow} X_{t} = X'\,,
\end{equation}
such that
\begin{enumerate}
\item[(a)] $X' = (X')^{\cS'}$,
\item[(b)] $\s$ is an isomorphism over $X^{\cS'}$.
\end{enumerate}
Moreover, the morphism $\s = \s_X$ (or the entire blowing-up sequence
\eqref{eq:blupX.4}) can be realized in a way that is functorial with respect
to \'etale morphisms.
\end{theorem}

Again the case $n=4$ is proved in \cite{BLM}. See \S\ref{subsec:inv.1} for the 
case $n=3$. 
As before, $\cS'$ is the smallest class of singularities satisfying the theorem.
The exceptional singularity is a limit of dpp singularities that cannot be eliminated 
by blowings-up.
(See Lemma \ref{lem:pp}) below and \cite[Rmk.\,1.6]{BLM}).

\begin{remark}\label{rem:minbirat}
Resolution of singularities of an embedded hypersurface can be reformulated as
``log-resolution'' of singularities of a Weil divisor $D$ on a variety $Z$. Stated in this way, 
Question \ref{qu:min} is the formulation of Koll\'ar \cite{Ko}, where $\cS$ is the smallest
class of singularities that includes all normal crossings singularities and satisfies condition 
(2) of Question \ref{qu:min} for the support of the birational transform of $D$. 

In the case $n = \dim Z = 3$, $\cS$ is the \emph{unique} smallest class of singularities
satisfying this version of Question \ref{qu:min}, in the following sense. If $\Supp D$
has a pp singularity $z^2 + xy^2 = 0$, with respect to a coordinate chart $U$ of $Z$
at a point $a=0$, then any proper birational morphism $U' \to U$ which is an isomorphism
precisely over $U \setminus \{a\}$, factors through the blowing-up of $\{a\}$ (by the
universal-mapping property of blowing up).

Likewise for $n=4$. In the case of a cyclic point singularity cp3; i.e., a hypersurface
$X \subset Z$ defined in local coordinates by $z^3 + wy^3 + w^2x^3 -3wxyz = 0$,
any birational morphism $Z' \to Z$, which modifies the cp3 singularity but is an isomorphism over $Z\setminus \{\text{dpp}, \text{cp3}\}$, factors through the blowing-up either of
$\{\text{cp3}\}$ or of $\{\text{dpp}, \text{cp3}\} = (z=y=w=0)$. 
But both of these blowings-up produce a new cp3 singularity.
\end{remark}

\begin{remark}\label{rem:totaltransf}
Our proof of
Theorem \ref{thm:min} also gives normal forms or local models for the singularites of the
total transform of $X$, corresponding to $\cS$. (Equivalently, it gives local models for the
``transform'' of a divisor $D$, where the latter is defined 
as the support of the birational transform
plus the exceptional divisor). For example, in the case $n=3$, the following
table gives the possible (reduced) exceptional divisors.
\medskip

\renewcommand{\arraystretch}{1.1}
\begin{tabular}{r | l}
               singularity\,\, & \,exceptional divisor\\\hline
                   $x = 0$ & $(y=0)$\\
                                & $(y=0) + (z=0)$\\
                                & $(x + z^2 = 0)$\\
                 $xy = 0$ & $(z=0)$\\
               $xyz = 0$ & \\
   $z^2 + xy^2 = 0$ & (x=0)
\end{tabular}
\medskip

The third line in the table gives the possibility of a non-transverse exceptional divisor
at a smooth point of the variety. This cannot be eliminated because it occurs in a
neighbourhood of the origin in the last line (pp).
\end{remark}

Following is a theorem on resolution except for codimension one singularities 
which can be
eliminated by normalizing. It can be considered also as  a ``higher-dimensional version'' 
of Theorems \ref{thm:min}, \ref{thm:minstrong} in the case $n=3$.

\begin{theorem}\label{thm:high3}
Let $X$ denote a reduced variety (in any dimension) 
and let $X^{\mathrm{ncp}}$ denote the open subset
of $X$ consisting of smooth points, double normal crossings points ($xy=0$) and
pinch points ($z^2 + xy^2 = 0$). Then there exists a morphism $\s: X' \to X$ which is
a finite composite of admissible blowings-up, such that
\begin{enumerate}
\item $X' = (X')^{\mathrm{ncp}}$;
\item $\s$ is an isomorphism over $X^{\mathrm{ncp}}$; 
\item $\Sing X'$ maps birationally onto the closure of $\Sing X^{\mathrm{ncp}}$.
\end{enumerate}
\end{theorem}

Again the theorem can be realized functorially, and it is easy to write local models
for the singularities of the total transform. Koll\'ar proves the assertion of Theorem
\ref{thm:high3} with a proper birational morphism $\s$ \cite[Thm.\,16]{Ko}. Note that
the term ``pinch point'' in Theorem \ref{thm:high3} means a hypersurface
singularity of the form $z^2 + xy^2 = 0$ in any number of variables $x,y,z,\ldots\,$.

Note that, in Theorem \ref{thm:min} in the case $n=3$, pinch points are isolated.
Theorem \ref{thm:high3} has an important new feature (which also occurs in the
case $n=4$ of Theorem \ref{thm:min}) --- any new singularities that occur as limits
of pinch points can be eliminated. Of course, Theorem \ref{thm:min} for $n=4$ 
suggests an analogue of Theorem \ref{thm:high3} with normal crossings singularities 
of order up to 3; we have not yet been able to prove this.

One can ask whether there are interesting relationships between the questions
above and other classification problems in singularity theory. For example, the
singularities in $\cS$ when $n=3$ are the same as those with occur for the images
of stable differentiable mappings $\vp: M^2 \to N^3$ (between manifolds of the
dimensions indicated). This question reflects a point of view towards resolution
of singularities suggested to us many years ago by Ren\'e Thom.

\subsection{The desingularization invariant as a computational tool}\label{subsec:inv.1}
Our proofs of the results in this article are based on using the desingularization invariant 
of \cite{BMinv} as a tool for computing and simplifying local normal forms. As an
illustration, we will outline proofs of Theorems \ref{thm:min} and \ref{thm:minstrong} 
in the case $n=3$ in this subsection. 

In the Appendix, we will try to provide a working knowledge of the desingularization
algorithm and the invariant as they are used here, for a reader not necessarily familiar
with a complete proof of resolution of singularities. Suppose that $X \hookrightarrow Z$
is an embedded hypersurface, where $Z$ is smooth. Let $\inv = \inv_X$ denote the
desingularization invariant for $X$. We recall that $\inv$ is defined iteratively on the
strict transform $X_{j+1}$ of $X =X_0$ for any finite sequence of $\inv$-\emph{admissible} 
blowings-up 
\begin{equation}\label{eq:finblup}
Z = Z_0 \stackrel{\s_1}{\longleftarrow} Z_1 \longleftarrow \cdots
\stackrel{\s_{j+1}}{\longleftarrow} Z_{j+1}\,.
\end{equation}
(A blowing-up is $\inv$-\emph{admissible} if it is admissible and $\inv$ is locally
constant
on its centre.) In particular, $\inv(a)$, where $a \in X_{j+1}$ depends not only on $X_{j+1}$
but also on the \emph{history} of blowings-up (\ref{eq:finblup}).

Let $a \in X_j$. then $\inv(a)$ has the form
\begin{equation}\label{eq:invintro}
\inv(a) = (\nu_1(a), s_1(a), \ldots, \nu_t(a), s_t(a), \nu_{t+1}(a))\,,
\end{equation}
where $\nu_k(a)$ is a positive rational number (``residual multiplicity'') if $k\leq t$, 
each $s_k(a)$ is a nonnegative integer (which counts certain components of the exceptional
divisor), and $\nu_{t+1}(a)$ is either $0$ or $\infty$.
The successive pairs $(\nu_k(a),s_k(a))$ are defined inductively over
\emph{maximal contact} subvarieties of increasing codimension.

Let $\inv_k$ denote the truncation of $\inv$ after the $k$'th pair in (\ref{eq:invintro})
($\inv_k := \inv$ if $k > t$).
Then $\inv_k$ can be defined iteratively over a sequence of $\inv_k$-admissible
blowings-up (\ref{eq:finblup}). For each $k$, $\inv_k$ is upper semicontinuous, and
also \emph{infinitesimally upper-semicontinuous} in the sense that $\inv_k$ can only
decrease after blowing up with $\inv_k$-admissible centre.

It is easy to see that, in \emph{year zero} (i.e., if $j=0$), then $\inv(a) = (2,0,1,0,\infty)$
if and only if $X$ has a double normal crossings singularity $z^2 + y^2 = 0$ at $a$.
Some other year-zero hypersurface examples:
\medskip

\renewcommand{\arraystretch}{1.15}
\begin{tabular}{r c l}
                   $x = 0$ & \,smooth\, & $\inv(0) = \inv(\mathrm{nc}1) := (1,0,\infty)$\\
               $x_1x_2\cdots x_k = 0$ & nc$k$ & 
                                  $\inv(0) = \inv(\mathrm{nc}k) := (k,0,1,0,\ldots,1,0,\infty)$\\
   $z^2 + xy^2 = 0$ & pp & $\inv(0) = \inv(\mathrm{pp}) := (2,0,3/2,0,1,0,\infty)$
\end{tabular}
\medskip

\noindent
(where, for nc$k$, there are $k-1$ pairs $(1,0)$). For $k \geq 3$, nc$k$ is 
not characterized by the value of $\inv$; for example, the singularity 
$x_1^k + x_2^k + \cdots x_k^k = 0$ also has $\inv(0) = (k,0,1,0,\ldots,1,0,\infty)$
with  $k-1$ pairs $(1,0)$.

Consider the case $\dim Z = 3$ and now suppose that $a \in X_j$, \emph{for an
arbitrary year} $j$. Then $\inv(a) = \inv(\mathrm{nc}2) = (2,0,1,0,\infty)$ 
if and only if $X_j$ can be
defined near $a$ by an equation
\begin{equation}\label{eq:ncp}
z^2 + x^\al y^2 = 0,
\end{equation}
where $\al$ is a positive integer and $(x=0)$ is an \emph{exceptional divisor}.
If we blow up with centre $z = x = 0$, the strict transform $X_{j+1}$ of $X_j$ is
given locally by
$$
z^2 + x^{\al -2} y^2 = 0 . 
$$
(The preceding blowing-up is $\inv_1$-admissible). After finitely many such blowings-up,
we get either $\al = 0$ or $\al = 1$; i.e., we get either
\begin{align*}
z^2 + y^2 &= 0 \quad \mathrm{nc}2\\
\mathrm{or} \quad z^2 + x y^2 &=0 \quad \mathrm{pp}.
\end{align*}

A crucial point is that the blowings-up we have described locally above are
actually globally-defined $\inv_1$-admissible blowings-up; the ideal $(x^\al)$
is the \emph{monomial part} of a \emph{coefficient marked ideal} defined on
a \emph{maximal contact hypersurface} (($z=0$) at the point $a$ above), at any
point of $(\inv_1 = (2,0)) := \{p: \inv_1(p) = (2,0)\}$. The blowings-up above, to 
reduce $\al$ to $0$ or $1$, constitute ``combinatorial or monomial
resolution of singularities''
of the monomial marked ideal. We will call this simplification of (\ref{eq:ncp}) by
resolution of singularities of the monomial marked ideal a \emph{cleaning} or
an application of the \emph{cleaning lemma} (see Section \ref{sec:clean}). 
Note the blowings-up
involved in applying the cleaning lemma above are $inv_1$- but not $\inv$-admissible.
See the Appendix for details of the ideas above.

\begin{proof}[Proof of Theorems \ref{thm:min} and \ref{thm:minstrong}, case $n=3$.]
We will first show that, given a reduced variety $X$ of dimension $2$, there exists 
a morphism $\s: X' \to X$ which is a finite composite of admissible blowings-up, 
as required in Theorem \ref{thm:min}.

Singular points of type nc3 are isolated, and each have only nc2 singularities in
some neighbourhood. Therefore, the points in the complement of $\{$nc3$\}$ with
$\inv > \inv(\mathrm{nc}2) = (2,0,1,0,\infty)$ form a closed set disjoint from $\{$nc3$\}$. 
So we can blow up with closed $\inv$-admissible centres with $\inv > \inv(\mathrm{nc}2)
= (2,0,1,0,\infty)$, until the maximum value of the invariant over the complement
of $\{$nc3$\}$ is $\leq \inv(\mathrm{nc}2)$. 

We now apply the cleaning lemma to blow up until every point of the stratum
$(\inv = \inv(\mathrm{nc}2))$ is either nc2 or pp. (There are now no other singular points
in a neighbourhood of this stratum.) The centres of blowing up involved in using the
cleaning lemma are disjoint from a neighbourhood of $\{$nc3$\}$.

We can now use the desingularization algorithm to resolve any singularities (i.e., to 
reduce to $\ord_{\cdot}X = 1$) in the
complement of $(\inv = \inv(\mathrm{nc}2))$ and the original $\{$nc3$\}$, by
admissible blowings-up. This suffices
to prove Theorem \ref{thm:min} in the case $n=3$. 

In order to prove Theorem \ref{thm:minstrong} in the case $n=3$ (in particular, to show that $\cS = \cS'$ in this case),
note that, if $n=3$, then singular points of type pp (as well as nc3) are isolated, and each 
have only nc2 singularities in some neighbourhood. So we can simply repeat the proof above,
changing ``$\{$nc3$\}$'' to ``$\{$nc3, pp$\}$'' wherever the former occurs. 
\end{proof}

\begin{remark}\label{rem:totalsings}
The argument above provides the normal forms listed in Remark \ref{rem:totaltransf}
for the total transform of $X$ at every singular point (i.e., nc3, pp, or nc2) of the total
transform. If, in addition to Theorems \ref{thm:min} and \ref{thm:minstrong} 
in the case $n=3$ , we want to get the normal forms for the
total transform listed in Remark \ref{rem:totaltransf} at all points, we may need to make
additional blowings-up of smooth points of the latter. 
The argument is similar to that above,
but we defer it to Section \ref{sec:pp} in order to take advantage of notions introduced
in Sections \ref{sec:clean} and \ref{sec:nc}.
\end{remark}

\subsection{Comparison with the desingularization algorithm}\label{subsec:ex}
The following example illustrates the way the cleaning lemma is used above, in
comparison with the ``monomial case'' of the desingularization algorithm (see
\cite[p.\,628]{BMfunct}).

\begin{example}\label{ex:comp} 
Let $X \subset \IA^3$ denote the hypersurface $(z^2 + x^3 y^2 = 0)$. We first
consider the desingularization algorithm applied to $X$. (A reader unfamiliar with
the computations below can refer to the Appendix.) 
\smallskip

\noindent
\emph{Year zero.} $\inv(0) = (2,0,5/2,0,1,0,\infty)$. The centre $C_0$ of the first blowing-up
$\s_1$ is $\{0\}$.
\smallskip

\noindent
\emph{Year one.} The total transform of $X_0 = X$ in the $x$-\emph{coordinate chart} (the chart given by substituting $(x, xy, xz)$ in place of $(x,y,z)$) is 
$$
x^2(z^2 + x^3y^2) = 0.
$$
(For simplicity of notation, we are again writing $(x,y,z)$ for the coordinates
after blowing up.) The strict transform $z^2 + x^3y^2 = 0$ has the same singularity at the 
origin as in year zero.
Now, however, $\inv(0) = (2,0,1,1,1,0,\infty)$. The centre $C_1$ of the blowing-up $\s_2$
is again $\{0\}$.
\smallskip

\noindent
\emph{Year two.} The total transform in the $x$-chart is given by
$$
x^4(z^2 + x^3y^2) = 0.
$$
The strict transform again has the same singularity at $\{0\}$! Now, however,
$\inv(0) = (2,0,1,0,\infty)$ and the next centre of blowing-up $C_2$ is the $x$-axis
$(z=y=0)$. Note that $C_2$ is nc2 when $x \neq 0$.
\smallskip

\noindent
\emph{Year three.} The total transform in the $y$-chart (given by the substitution
$(x,y,yz)$) is
$$
x^4y^2(z^2 + x^3) = 0.
$$
We are now in the monomial case of the desingularization algorithm; $\inv(0)
= (2,0,0)$. The next blowing-up (centre $(z=x=0)$) resolves the singularities of
the strict transform, but additional blowings-up are needed to make the total
transform simple normal crossings.
\smallskip

By comparison, if we follow the algorithm involved in the proof of Theorem
\ref{thm:min} in \S\ref{subsec:inv}, we would take a different route starting
in year two, when $\inv(0) = (2,0,1,0,\infty)$: We would apply the cleaning
lemma, which tells us to blow up with centre $C'_2 := (z=x=0)$ (as in year
three above). We then get a pinch point
$$
x^6(z^2 + xy^2) = 0
$$
(without blowing up nc2 points). The cleaning lemma is essentially the
monomial case of resolution of singularities, but (with reference to the
algorithm of \cite{BMinv, BMfunct}) applied at an intermediate step, rather
than after reduction to the monomial case.
\end{example}

\section{Cleaning Lemma}\label{sec:clean}

Let $\ucI = (Z,N,E,\cI,d)$ denote a marked ideal (see \S\ref{subsec:markedideal}) and
let $\cI = \cM(\ucI)\cdot\cR(\ucI)$ denote the factorization of $\cI$ into monomial
and residual parts (see \S\ref{subsec:monres}). The ideal $\cM(\ucI)$ is locally generated by a monomial
in components of the exceptional divisor, whose exponents divided by $d$ are invariants
of the equivalence class of $\ucI$ (see Definition \ref{def:invts}).
Set $\ucM(\ucI) = (\cM(\ucI), d)$. Then $\cosupp \ucM(\ucI)
\subset \cosupp \ucI$ and any admissible sequence of blowings-up of $\ucM(\ucI)$ is 
admissible for $\ucI$.

In general, however, it is not true that the transforms $\ucI'$ and $\ucM(\ucI)'$ by an
admissible blowing-up of $\ucM(\ucI)$ satisfy $\ucM(\ucI') = \ucM(\ucI)'$ (since 
exceptional divisors might factor from the pull-back of $\cR(\ucI)$).

\begin{lemma}\label{lem:clean}
Suppose that $\ucR(\ucI) := (Z,N,E,\cR(\ucI), \ord\,\cR(\ucI))$ admits a maximal contact
hypersurface $P$ at some point of its cosupport (in particular, $E$ is 
transverse to $P$; see Definitions \ref{def:transverse}, \ref{def:max}).
Then, after transformation of $\ucI$ by an admissible sequence of blowings-up of $\ucM(\ucI)$,
we can assume that $\cosupp \ucM(\ucI)$ is disjoint from the strict transform of $P$.
\end{lemma}

\begin{proof}
Any non-empty intersection $D$ of components of $E$ is transverse to $P$.
Therefore, if we blow up with centre $C = D \cap N$, then, on the strict
transform of $P$, no exceptional
divisor factors from the pull-back of $\ucR(\ucI)$ and the transforms of 
$\ucI'$ and $\ucM(\ucI)'$
satisfy $\ucM(\ucI') = \ucM(\ucI)'$. The result 
follows from desingularization of $\ucM(\ucI)$.
\end{proof}

\subsection{Cleaning}\label{subsec:clean}
Consider the desingularization algorithm for an embedded hypersurface 
$X \hookrightarrow Z$, and let $a \in X_{j_0}$, for some $j=j_0$ (notation of
\ref{subsec:inv}). 
The invariant
$\inv(a)$ has the form
$$
\inv(a) = (\nu_1(a), s_1(a), \ldots, \nu_q(a), s_q(a), \nu_{q+1}(a))
$$ 
(see (\ref{eq:inv})). Suppose that $p < q$.
According to the desingularization algorithm,
$\Sigma_p := (\inv_p \geq \inv_p(a))$ is (locally) the support of a marked ideal 
$\ucI = \ucI^p = (Z_{j_0},N,E,\cI,d)$ on a maximal contact subvariety $N$ of codimension $p$ in 
$Z_{j_0}$. Consider $\ucM(\ucI^p)$ and $\ucR(\ucI^p)$ as above. 
Then there is an $\inv_p$-admissible sequence of blowings-up of $Z_{j_0}$, 
\begin{equation}\label{eq:cleanblup}
Z = Z_{j_0} \stackrel{\s_{j_0 +1}}{\longleftarrow} Z_{j_0 +1} \longleftarrow \cdots
\stackrel{\s_{j_1}}{\longleftarrow} Z_{j_1}\,,
\end{equation}
such that $\cosupp \ucM(\ucI^p)_{j_1} = \emptyset$,
where $ \ucM(\ucI^p)_{j_1}$ denotes the transform of $\ucM(\ucI^p)$ in year $j_1$
(by resolution of singularities of a monomial marked ideal 
\cite[Sect.\,5, Step II.A]{BMfunct}. 
Such a blowing-up sequence will be called a \emph{cleaning}. The centres of 
blowing up involved are invariantly defined (cf. Definition \ref{def:invts}). 

As remarked above, cleaning does not necessarily mean that $\cosupp\allowbreak \ucM(\ucI_{j_1})
= \emptyset$, though $\cosupp \ucM(\ucI_{j_1})$ will be disjoint from points where
Lemma \ref{lem:clean} applies.

In general, we will use cleaning to transform $\cosupp \ucM(\ucI^p)$ to $\emptyset$,
successively for $p = q-1,\, q-2, \ldots, 1$.

\begin{example}\label{ex:clean}
Suppose that $s_{p+1}(a) = 0$. Then $E$ consists of only ``new'' exceptional
divisors for $\inv_{p+ 1/2}$ at $a$ (see \S\S\ref{subsec:bdry},\,\ref{subsec:invgeneral}), 
so that $\ucR(\ucI^p)$ has a maximal contact
hypersurface in $N$ transverse to $E$.
\end{example}

\begin{remark}\label{rem:clean}
The truncated invariant $\inv_p$ is well-defined over \eqref{eq:cleanblup}, and is
both semicontinuous and infinitesimally semicontinuous (i.e., it can only decrease
after each blowing-up). But the cleaning sequence \eqref{eq:cleanblup} is not, in
general, $\inv$-admissible. The residual multiplicity $\nu_{p+1}$ (see Definition
\ref{def:invgeneral}) can be
defined as usual over \eqref{eq:cleanblup}, so that $\inv_{p+ 1/2}$ is well-defined
and semicontinuous (though not
necessarily infinitesimally semicontinuous). 

Moreover, we can extend $\inv_{p+ 1/2}$ to a modified invariant 
on $Z_{j_1}$ by considering $j_1$ to be ``year zero'' for $\inv_{p+ 1/2}$,
and can then follow the usual desingularization algorithm and definition 
of $\inv$ starting in this year (i.e., $j_1$ will be the year of birth for 
the value $\inv_{p+ 1/2}(a)$ of $\inv_{p+ 1/2}$ over a point $a \in Z_{j_1}$).
In other words, all components of the exceptional divisor
at $a$, except those counted by $s_1(a),\dots,s_p(a)$ are counted by 
$s_{p+1}(a)$ (they are considered the ``old'' exceptional divisors for 
$\inv_{p+ 1/2}$ at $a$); then 
$\inv_{p+1} := \left(\inv_{p+ 1/2}, s_{p+1}\right)$ extends to a 
semicontinuous invariant $\inv$ on $Z_{j_1}$ by the construction in 
\S\ref{subsec:invgeneral}, and we can afterwards follow the
desingularization algorithm.
\end{remark}

\section{Simple normal crossings}\label{sec:nc}
In this section, we prove the following result (see Theorem \ref{thm:nc}).

\begin{theorem}\label{thm:snc}
Let $X$ denote a reduced variety and let $X^{\mathrm{snc}}$ denote the
locus of points of $X$ have only simple normal crossings singularities.
Then there is a morphism $\s: X' \to X$ which is a composite of finitely
many admissible blowings-up, such that
\begin{enumerate}
\item $X' = (X')^{\mathrm{snc}}$;
\item $\s$ is an isomorphism over $X^{\mathrm{snc}}$;
\item $\s$ maps $\Sing X'$ birationally onto the closure of $\Sing X^{\mathrm{snc}}$.
\end{enumerate}
\end{theorem}

\begin{remark}\label{rem:nc}
Let $X^{\mathrm{snc}{\leq}r}$ denote the
locus of points of $X$ having only simple normal crossings singularities of orders
$\leq r$. There is a simple variant of Theorem \ref{thm:snc} where snc is replaced
by snc$\leq r$. For example, we can deduce this from Theorem \ref{thm:snc} by
blowing up singularities of order $> r$.
\end{remark}

\begin{definitions}\label{def:transverse}
Let $X \hookrightarrow Z$, where $Z$ is smooth
of dimension $n$. Let $E$ denote a finite collection of smooth hypersurfaces in $Z$
having only simple normal crossings. We say that $(X,E)$ is \emph{simple normal crossings}
(\emph{snc}) at a point $a$ if there is a regular system of parameters $(x_1,\ldots,x_n)$ at $a$ 
in which each irreducible component of $X$ is a coordinate subspace and each member
of $E$ is a coordinate hyperplane. There is an analogous notion of \emph{normal crossings}
(\emph{nc}) at $a$. We say that $X$ and $E$ are \emph{transverse} at $a$ if
they are nc and each component of $E$ is transverse to $X$ at $a$. 
We write $(X,E)^{\mathrm{snc}}$ to denote the simple normal crossings
locus of $(X,E)$.
\end{definitions}

Consider a sequence of blowings-up of $Z$,
\begin {equation}\label{eq:ambblup}
Z = Z_0 \stackrel{\s_1}{\longleftarrow} Z_1 \longleftarrow \cdots
\stackrel{\s_t}{\longleftarrow} Z_t \,.
\end{equation}
Write $X_0 := X$ and $E_0 := E$, where we 
order the members of $E_0$ in an arbitrary way. Let $X_{j+1}$ denote the
strict transform of $X_j$, $j=0,1,\ldots,$. We again say that the sequence (\ref{eq:ambblup}) 
is \emph{admissible} if, for each 
successive $j=0,1,\ldots,$ the blowing-up
$\s_{j+1}$ has smooth centre $C_j \subset X_j$ such that
$(C_j,E_j)$ is snc, where, for all $j\geq 1$, $E_{j}$ denotes the (ordered) collection of
strict transforms of the members of $E_{j-1}$, together with $\s_j^{-1}(C_{j-1})$
added as the last element. 

\begin{theorem}\label{thm:gennc}
Let $X \hookrightarrow Z$ denote an embedded reduced hypersurface, where $Z$ is
smooth, and let $E$ denote a finite collection of smooth hypersurfaces in $Z$
having only simple normal crossings. Then
there is a finite admissible sequence of blowings-up (\ref{eq:ambblup})
such that
\begin{enumerate}
\item $(X_t, E_t) = (X_t, E_t)^{\mathrm{snc}}$;
\item the morphism $\s$ given by the composite of the $\s_j$ is
an isomorphism over $(X,E)^{\mathrm{snc}}$;
\item $\s$ maps $\Sing X_t$ birationally onto the closure of $\Sing X^{\mathrm{snc}}$
\end{enumerate}
\end{theorem}

Moreover, the theorem is functorial with respect to \'etale or smooth morphisms
preserving the number of irreducible components of $X$ and $E$ at every point (cf.
Remark \ref{rem:funct}).

The sequence of blowings-up (\ref{eq:ambblup}) will be independent of an
ordering of $E_0$. If $X \hookrightarrow Z$ is an embedded variety, then the strong desingularization algorithm of \cite{BMinv, BMfunct} (cf. Remark \ref{rem:min}(2)) proceeds 
by first blowing up non-hypersurface points and points where (the transform
of) $E$ is not tranverse to a local minimal embedding variety for (that of) $X$,
to reduce to the case that $X \hookrightarrow Z$ is an embedded hypersurface.
So we can reduce Theorem \ref{thm:snc} to Theorem \ref{thm:gennc}. On the other
hand, we can reduce Theorem \ref{thm:gennc} to the case that $E = \emptyset$,
simply by replacing $X$ by $X \cup E$.

Let $X \subset Z$ be as in Theorem \ref{thm:gennc} (with $E = \emptyset$). Consider the
desingularization invariant $\inv = \inv_X$ and the sequence of $\inv$-admissible
blowings-up (\ref{eq:ambblup}) given by the desingularization algorithm of \cite{BMinv, BMfunct}. Let $a \in X_j$. We will write $a_i$ to denote the image of $a$ in $X_i$, for
any $i \leq j$.

Recall that, if $X$ is nc$q$ at a point $a$, then $\inv(a) = \io_q$, where
$$
\io_q := (q,0,1,0,\ldots,1,0,\infty)
$$
with $q-1$ pairs $(1,0)$.

\begin{lemma}\label{lem:nc}
Let $X \subset Z$ denote an embedded hypersurface, where $Z$ is smooth.
Consider the
desingularization invariant $\inv = \inv_X$ and the sequence of $\inv$-admissible
blowings-up (\ref{eq:ambblup}) given by the desingularization algorithm, as above. 
\begin{enumerate}
\item Let $a \in X = X_0$. Then $\inv(a) = \io_q$ and $X$ has $q$ local
analytic (respectively, irreducible) components at $a$ if and only if $X$ is nc$q$ 
(respectively, snc$q$) at $a$.
\item Let $a \in X_j$, for given $j$. If $\inv(a) = \io_q$ and $X_j$ has $q$ local
analytic (respectively, irreducible) components at $a$, then we can choose
local analytic, i.e., \'etale (respectively, regular) coordinates at $a$, 
$$
(x,u) = (x_1,\ldots,x_q, u_1,\ldots, u_{n-q}),
$$
in which the ideal of $X_j$ is generated by a product $f = f_1\cdots f_q$ such that
\begin{align*}
f_1 &= x_1,\\
f_2 &= x_1 + u^{\al^1}x_2,\\
f_3 &= x_1 + u^{\al^1}\left( x_2\cdot\xi_{32} + u^{\al^2}x_3\right),\\
f_4 &= x_1 + u^{\al^1}\left( x_2\cdot\xi_{42} + u^{\al^2}\left(x_3\cdot\xi_{43}
+ u^{\al^3}x_3\right)\right),\\
&\cdots,
\end{align*}
where each
$u^{\al^k} = u_1^{\al_1^k}\cdots u_{n-q}^{\al_{n-q}^k}$, with $(u_l = 0) \in E_j$ if 
$\al_l^k > 0$.
\end{enumerate}
\end{lemma}

\begin{proof} 
See \S\ref{subsec:inv} for the ``if'' direction of (1). The ``only if'' direction of (1)
is a special case of (2). We will prove (2).

By the Weierstrass preparation theorem, the ideal of $X_j$ at $a$ has a generator 
of the form
\begin{equation}\label{eq:wform}
f(y,z) = z^q + \sum_{i=2}^q b_i(y)z^{q-i}
\end{equation}
in local \'etale coordinates $(y,z) = (y_1,\ldots,y_{n-1}, z)$ at $a=0$, where
$\ord_a b_i \geq i$, for each $i$, and $(z=0)$ is a maximal contact hypersurface.
Since $X_j$ has $q$ components at $a$, we can factor \eqref{eq:wform} as
\begin{equation*}
z^q + \sum_{i=2}^q b_i(y)z^{q-i} = \prod_{j=1}^q \left(z-a_j(y)\right),
\end{equation*}
where $\sum a_j = 0$.

Then the coefficient ideal corresponding to the maximal contact hypersurface
$(z=0)$ is equivalent to $((a_j), 1)$. (See Example \ref{ex:split}.) 

Since $\inv(a) = \io_q$, the ideal $(a_j)$ has order $1$ at $a$, after division by a
monomial $u^{\al^1}$ in the exceptional divisor. After a change of coordinates,
we can assume that $a_1 = u^{\al^1}y_1$, where $(z=y_1=0)$ is a second
maximal contact subspace (i.e., maximal contact subspace of codimension $2$), and
that each $a_j$, $j\geq 2$, is of the form
\begin{equation*}
a_j = u^{\al^1}\left( y_1\cdot\eta_{j1} + c_j\right).
\end{equation*}

Again, the ideal $(c_j)$ on $(z=y_1=0)$ has order $1$ at $a$ after division by
an exceptional monomial $u^{\al^2}$, and so on. So we can write $f$ in the form
$f = f_1\cdots f_q$, where the first $q-1$ factors are of the form
\begin{align*}
f_1 &= z + u^{\al^1}y_1,\\
f_2 &= z + u^{\al^1}\left( y_1\cdot\eta_{21} + u^{\al^2}y_2\right),\\
f_3 &= z + u^{\al^1}\left( y_1\cdot\eta_{31} + u^{\al^2}\left(y_2\eta_{32} 
+ u^{\al^3}y_3\right)\right),\\
&\cdots,
\end{align*}
(recall that $\sum a_j = 0$) and the result follows, by a further coordinate change.
\end{proof}

\begin{proof}[Proof of Theorem \ref{thm:gennc}]
We can assume that $E=\emptyset$.
Given $p \in \IN$, let $\Sigma_p(X)$ denote the locus of points lying in at least $p$ 
irreducible components of $X$.
Let $q$ denote the largest value of $\ord_aX$ at snc points of $X$. We blow up with 
$\inv$-admissible centres following the desingularization algorithm as long as the 
maximum value of $\inv$ is $> \io_q$, stopping when
the maximum value $= \io_q$, say in year $j_0$. Set $I_q(X,j_0) := (\inv = \io_q)
\subset X_{j_0}$. Then $I_q(X,j_0) \neq \emptyset$ since it includes the snc points of 
$X$ (includes in the sense that all previous blowings up are isomorphisms over such 
points of $X$).

Using the desingularization algorithm, we can blow up any component of $I_q(X,j_0)$
which is not generically snc (to decrease $\inv$). Therefore, we can assume that every
component of $I_q(X,j_0)$ is generically snc.

Let $a \in I_q(X,j_0)$. Choose coordinates at $a$ satisfying Lemma \ref{lem:nc}.
The locus
$$
\left(x_1 = \cdots = x_{q-1} = 0\right) \bigcap 
\left(\ord\,u^{\al^{q-1}}\geq 1\right)
$$
is the cosupport of a monomial marked ideal of order $1$ on a maximal contact
subvariety $N = (x_1 = \cdots = x_{q-1} = 0)$ of codimension $q-1$. According to the Cleaning Lemma \ref{lem:clean}, we can reduce
$\al^{q-1}$ to $0$ by finitely many globally-defined $\inv_{q-1}$-admissible blowings-up.

We can repeat the preceding
argument using the monomial marked ideal $\left((u^{\al^{q-2}}),1\right)$ on the
subspace $x_q = x_1 = \cdots x_{q-2} = 0$ to reduce $\al^{q-2}$ 
to $0$, etc., eventually to reduce all $\al^k$ to $0$.

\begin{remark}\label{rem:functclean}
For simplicity, we have begun in a way that ignores the problem of functoriality. In fact, 
if $n := \dim Z$, then, for each $p = n,n-1,\ldots,q$, we should follow the desingularization
algorithm (starting as if in ``year zero'') until $\inv \leq \io_p$, even if 
$I_p(X,j) = \emptyset$ ($p > q$),  blow up any component of $I_p(X,j)$ which is not
generically snc, and then perform cleaning as above. Globally, 
$(x_1 = \cdots = x_{k} = 0) \cap (\ord\,u^{\al^k} \geq 1)$, $k \leq p-1$, 
is given by the cosupport of
an invariantly defined monomial marked ideal $\ucM(\ucI^{k})$ on the locus
$(\inv_{k} \geq (\io_p)_k)$ for the truncated invariant, and the cleaning
procedure of \S\ref{subsec:clean} applies. 
\end{remark}

Suppose that we are now in year $j_1$. The result of our cleaning above is that 
$(X_{j_1}, E_{j_1})$ is snc at all points of $\Sigma_q(X_{j_1})$, and therefore in a
neighbourhood of $\Sigma_q(X_{j_1})$.

We now apply the desingularization algorithm to $(X_{j_1}, E_{j_1})$ restricted to the
complement of $\Sigma_q(X_{j_1})$ (where we regard $j_1$ as ``year zero'') to blow up with
smooth centres over the complement of $\Sigma_q(X_{j_1})$ until the maximum value of $\inv$
is $\leq \io_{q-1}$. 

However, the centres of the blowings-up involved
will not necessarily be closed in $X_{j_1}$ and its strict transforms (since, in the process,
we will introduce nonzero $s$-terms in $\inv$).

For example, the total transform of $X$ at a point of $\Sigma_q(X_{j_1})$ is of the form
$(u^\al x_1 \cdots x_q = 0)$, where $u^\al = u_1^{\al_1} \cdots u_{n-q}^{\al_{n-q}}$ is a 
monomial in exceptional divisors. The  centre
of the blowing up of $X_{j_1}$ will be given near such a point of $\Sigma_q(X_{j_1})$ by
\begin{equation}\label{eq:snc}
\bigcup_{i=1}^q \left(u_{l_1} = \cdots = u_{l_p} = x_1 = \cdots = \widehat{x_i} 
= \cdots = x_q = 0,\,\, x_i \neq 0\right),
\end{equation}
for some $l_1, \ldots, l_p$ (where $\widehat{x_i}$ means that $x_i$ is deleted from the expression).

We can simply modify the algorithm by first blowing up with centre given by 
$$
\left(u_{l_1} = \cdots = u_{l_p} = x_1 = \cdots = x_q = 0\right)
$$
(the intersection of the closures of the components in (\ref{eq:snc})) to separate the
components, and by then blowing up the union of these (closed) components. The
two blowings-up are admissible and include no (points lying over) snc points of $X$.

In general, given a union of subvarieties
$$
\left(u_{l_{i1}} = \cdots = u_{l_{ip_i}} = x_1 = \cdots = \widehat{x_i} 
= \cdots = x_q = 0\right),
$$
for certain $i=1,\ldots,q$ (where each $p_i>0$), 
we can blow up finitely many times with centres of increasing
dimension in $\Sigma_q(X_j)$, $j = j_1,\ldots$, to separate these 
varieties (before blowing them up, for example).

We thus modify each of the blowings-up of $(X_{j_1}, E_{j_1})$ above; we get a finite
sequence of blowings-up with closed admissible centres over the complement of the snc locus
of $X$, after which $(X_j, E_j)$ is snc on $T_q(X_j)$, where $T_q(X_j)$ denotes the inverse
image of $\Sigma_q(X_{j_1})$ in $X_j$, and the maximum value of 
$\inv$ on the complement of $T_q(X_j)$ is $\leq \io_{q-1}$, for some $j = j_1' \geq j_1$).

We then blow up any component of $I_{q-1}(X,j_1')$ which is not generically snc, and
apply the Cleaning Lemma as above (over the complement of $T_q(X_{j_1'})$),
to blow up further until we have 
$(X_{j_2}, E_{j_2})$ (for some year $j_2$) snc at every point of $\Sigma_{q-1}(X_{j_2})$.
(The centres of the blowings-up involved will be separated
from the successive $T_q(X_j)$ because $(X_j, E_j)$ is already snc in a neighbourhood
of the latter.)

In general, suppose that, for some year $j_k$, $(X_{j_k}, E_{j_k})$ is snc on 
$\Sigma_{q-k+1}(X_{j_k})$. We apply the desingularization
algorithm over the complement of $\Sigma_{q-k+1}(X_{j_k})$ as above, 
until the maximum value
of $\inv$ is $\leq \io_{q-k}$. 
The closure of each centre of blowing up can be separated into
a disjoint union of smooth subvarieties as above. Afterwards, we again blow up the
components of $(\inv = \io_{q-k})$ that are not generically snc,
and then apply the Cleaning Lemma.
So we get a finite sequence of blowings-up
with smooth admissible centres, after which $(X_{j_{k+1}}, E_{j_{k+1}})$ has snc at every
point of $\Sigma_{q-k}(X_{j_{k+1}})$. 

Eventually, we get $(X_j, E_j)$ snc on $\Sigma_1(X_j) = X_j$.
We thus get the theorem with conditions (1) and (2), and condition (3) is
clear from the choices of blowings-up (see also Theorem \ref{thm:nc} ff.).
\end{proof}

\section{Pinch points}\label{sec:pp}
Our main goal in this section is to prove Theorem \ref{thm:high3}. In comparison
with Theorem \ref{thm:min} in the case $n=3$, the problem here is to eliminate
new singularities that intervene as limits of pinch points (\S\ref{subsec:pp}). 
By contrast, in \S\ref{subsec:dpp}, we will show
that new singularities which occur as limits of degenerate pinch points cannot
necessarily be eliminated.

Before turning to Theorem \ref{thm:high3}, we indicate how to get the normal forms
listed in Remark \ref{rem:totaltransf} for the total transform in Theorems
\ref{thm:min} and \ref{thm:minstrong} in the case $n=3$.

\subsection{Minimal singularities in $3$ variables}\label{subsec:pp}

Let $X \subset Z$ denote an embedded hypersurface where $Z$ is smooth and
of pure dimension $3$. According to the proofs of Theorems \ref{thm:min} and
\ref{thm:minstrong} in the case $n=3$ (see \S\ref{subsec:min}),
we have a sequence of blowings-up
$$
Z = Z_0 \stackrel{\s_1}{\longleftarrow} Z_1 \longleftarrow \cdots
\stackrel{\s_j}{\longleftarrow} Z_j \,
$$
after which every point of $X_j$ has only nc3, nc2 and pp singularities.
Moreover, we have the normal forms listed in Remark \ref{rem:totaltransf} at every
singular point of $X_j$ (see Remark \ref{rem:totalsings}).

\begin{remark}\label{rem:totalsingsn=3}
\emph{How to get the normal forms of Remark \ref{rem:totaltransf} at every point.}
Write $W := Z_j$, $Y :=X_j$, and let $E$ denote (the support of) the exceptional divisor $E_j$.
Set $\Sigma := \Sing Y$. 
We apply the desingularization algorithm to $(Y,E)$ in $W$, over the open subset
$V = W\setminus \Sigma$. This is now ``year zero'' for the desingularization algorithm,
so that $\inv$ will have a meaning different than before. For example, consider a pp
where $Y = (z^2 + xy^2 = 0)$ and $(x=0)$ is the exceptional divisor; then at a nearby
point $z = x = 0$, $y \neq 0$, we have $\inv = (1,1,2,0,\infty)$. There is a neighbourhood
of $\Sigma$ in which $Y \cap V$ has only smooth points, but $(Y,E)$ has the following
possible forms, characterized by the value of the invariant shown (in year zero).
\medskip

\begin{tabular}{l l l}
$\,\,Y:\, y=0\,\,$ & $\,\,E:\, \emptyset\,\,$ & $\,\,\inv = (1, 0, \infty)\,\,$\\
$\,\,Y:\, y=0\,\,$ & $\,\,E:\, x=0\,\,$ & $\,\,\inv = (1,1,1, 0, \infty)\,\,$\\
$\,\,Y:\, y=0\,\,$ & $\,\,E:\, y+x^2 = 0\,\,$  & $\,\,\inv = (1,1,2, 0, \infty)\,\,$
\end{tabular}
\medskip

We blow up with centre prescribed by the desinglarization algorithm for $(Y,E)$
restricted to $V$, until the maximum value of $\inv$ is $(1,1,2, 0, \infty)$. The centres
involved are separated from $\Sigma$ and its inverse images. We can also blow up
any closed component of $(\inv = (1,1,2, 0, \infty)$. 

Now, at a point where $\inv = (1,1,2, 0, \infty)$, the strict transform of $Y \cup E$ is 
given by an equation
$$
y(y+u^\al x^2) = 0,
$$
where $(u=0)$ is the exceptional divisor. We can blow up using the Cleaning Lemma
to reduce to $\al = 0$. (The centres of the blowings-up involved in cleaning are separated
from the inverse images of $\Sigma$.) We thus reduce to the case that the (strict transforms
of) $Y$, $E$ are given by equations of the form $y=0$, $y+x^2 = 0$ (respectively) at every
point of the transform of $(\inv = (1,1,2, 0, \infty))$.

Let $\Sigma'$ denote the union of the latter and the inverse image of $\Sigma$. We repeat
the argument above to blow up (over the complement of $\Sigma'$) until the maximum
value of $\inv$ is $(1,1,1, 0, \infty)$, and then use the Cleaning Lemma to reduce locally
to $Y = (y=0)$, $E = (x=0)$.

A further sequence of blowings-up over the complement of the points already considered,
until the maximum value of $\inv$ becomes $(1,0,\infty)$, completes the argument.
\end{remark}
  
\subsection{Pinch points in higher dimension}\label{subsec:highpp}
Consider a hypersurface $X \hookrightarrow Z$, $Z$ smooth, with a pinch point
singularity at a point $a$; in local coordinates, $z^2 +xy^2 = 0$. Then
\begin{equation}\label{eq:invpp}
\inv(a) = \left(2,0,3/2,0,1,0,\infty\right).
\end{equation}
But (\ref{eq:invpp}) does not guarantee that $a$ is a pinch point; for example,
$z^2 + y^3 +x^3 = 0$ has the same value of $\inv$ but an isolated singularity at $0$.

\begin{lemma}\label{lem:pp}
Let $X \hookrightarrow Z$ denote a hypersurface, $Z$ smooth, and let $a \in X$. Then
\begin{enumerate}
\item $a$ is a pinch point pp if and only if 
$$
\inv(a) = (2,0,3/2,0,1,0,\infty)
$$
and the singular subset of $X$, $\Sing X$ has codimension $2$ in $Z$ at $a$;
\item $a$ is a degenerate pinch point dpp if and only if 
$$
\inv(a) = (2,0,3/2,0,2,0,\infty)
$$  
and $\Sing X$ has codimension $2$ at $a$.
\end{enumerate}
\end{lemma}

\begin{proof}
Suppose that $X$ has order $2$ at a point $a$. Then, in suitable \'etale
local coordinates $(x,z) = (x_1,\ldots,x_{n-1},z)$ at $a$, $X$ is given by an equation
\begin{equation}\label{eq:ord2}
z^2 + b(x) = 0.
\end{equation}
If $\inv_2(a) = (2,0,3/2,0)$, then we can choose new coordinates $(x,y,z)
= (x_1,\ldots,x_{n-2},y,z)$ in which (\ref{eq:ord2}) becomes
\begin{equation}\label{eq:ord2+}
z^2 + y^3 + B(x)y + C(x) = 0.
\end{equation}
Then $\Sing X$ lies in 
\begin{align*}
z &= 0,\\
y^3 + B(x)y + C(x) &= 0,\\
3y^2 + B(x) &=0.
\end{align*}
If $\Sing X$ has codimension $2$ at $a$, then the last 2 equations have a common
factor, so (\ref{eq:ord2+}) can be rewritten in the form
\begin{equation}\label{eq:ord2f}
z^2 + (y - A(x))^2 (y + 2A(x)) = 0.
\end{equation}

Clearly, $a$ is a pinch point if and only if $\ord_a A =1$, and (1) follows.
Likewise $\inv(a) = (2,0,3/2,0,2,0,\infty)$ if and only if $A$ is the square of
a function of order $1$, so (2) follows.
\end{proof}

\begin{proof}[Proof of Theorem \ref{thm:high3}]
We can reduce to the case that $X \hookrightarrow Z$ is an embedded hypersurface,
$Z$ smooth. We then divide the argument into three parts:
\medskip

\noindent
(I) We can blow up following the desingularization algorithm as long as
the maximum value of $\inv$ is $> \inv(\mathrm{pp}) := (2,0,3/2,0,1,0,\infty)$. 
(The blowings-up involved do
not modify pp, nc2 or smooth points of $X$.)

Suppose that the maximum value of $\inv$ is $\inv(\mathrm{pp})$ (in some year of
the resolution history). Then the locus $(\inv = \inv(\mathrm{pp})$ is a smooth subset
of $X$ of codimension $3$ in $Z$. Each component of this set either contains no pp or
is generically pp (according as $\Sing X$ has codimension $> 2$ or $=2$ at the generic
point). We can blow up to get rid of all components with no pp.

Then at any point $a$ with $\inv(a) = \inv(\mathrm{pp})$, the strict transform of $X$
is defined by an equation
$$
z^2 + u^\al \left(y + u^\be x \right)^2 \left(y - 2u^\be x\right) = 0,
$$
in suitable \'etale local coordinates $(u,x,y,z) = (u_1,\ldots,u_{n-3},x,y,z)$ for $Z$,
where $u^\al = u_1^{\al_1}\cdots u_{n-3}^{\al_{n-3}}$ and $\al_i > 0$ only if $(u_i=0)$ is
a component of the exceptional divisor (and likewise for $u^\be$).

We use the cleaning lemma first to reduce to the case $\be = 0$
($\al$ will increase in the process): 
\begin{equation}\label{eq:ppclean}
\left(z=y=0, \, \ord\,  u^\be \geq 1\right) \subset \left(\inv_2 = (2,0,3/2,0)\right)
\end{equation}
is the cosupport of an invariantly defined monomial marked ideal with associated
multiplicity $1$ on a maximal contact subvariety of codimension $2$; any component
of this set extends to an $\inv_2$-admissible centre of blowing up. The blowings-up
involved in applying the cleaning lemma have centres given locally by components
of (\ref{eq:ppclean}) and its transforms.

Secondly, we use the cleaning lemma to reduce to the case $|\al| \leq 1$, where
$|\al| := \al_1 + \cdots \al_{n-3}$, using 
$$
\left(z=0, \, \ord\,  u^\al \geq 2\right) \subset \left(\inv_1 = (2,0)\right).
$$

If $\al = 0$, we have a pinch point. If $|\al| = 1$,then we have a singularity of the
form
$$
z^2 + u_1(y+x)^2(y-2x) = 0,
$$
where $u_1$ is an exceptional divisor. In this case, we blow up with centre given locally by
$$
\left(z = y = x = u_1 =0\right) \subset \left(\inv = (2,0,3/2,0,1,0,\infty)\right)
$$ 
to get
$$
z^2 + u_1^2(y+x)^2(y-2x) = 0.
$$
We now repeat the second cleaning step above to get a pinch point.
\medskip

\noindent
(II) Let us say we are now in year $j_0$. At any point of the pp locus, we can
choose \'etale coordinates in which $X$ and the support of the exceptional divisor
are given as
\begin{equation} \label{eq:pploc}
z^2 + xy^2 = 0 \text{ and } \sum_{i=0}^s (u_i = 0)  
\end{equation}
(respectively), for some $s \geq 0$. At nearby nc2 singularities (when $x \neq 0$
above), we can find \'etale coordinates in which $X$ and the support of the exceptional 
divisor are given as
\begin{equation} \label{eq:nc2loc}
z^2 + y^2 = 0 \text{ and } \sum_{i=0}^s (u_i = 0),
\end{equation}
for some $s \geq 0$.

We now apply the desingularization algorithm outside the pp locus (where we consider
$j_0$ as ``year zero'') until the maximum value of $\inv$ is 
$\inv(\text{nc}2) = (2,0,1,0,\infty)$. 
Each component of every centre of blowing up involved is either separated from (the inverse image of) the pp locus above, or, near a point as in \eqref{eq:pploc}, of the form 
$z=y=0,\, u_j = 0$, for certain $j$. We can handle this as in the proof of Theorem
\ref{thm:gennc}, blowing up to separate such components at the pp locus before we
blow them up.  

Cleaning as in the proof of Theorem \ref{thm:min}, case $n=3$ (see \S\ref{subsec:min}), 
produces nc2, or pp at special points of the stratum $(\inv = (2,0,1,0,\infty))$.

\medskip

\noindent
(III) We can now use the desingularization algorithm
to resolve any singularities remaining outside $\{$nc2, pp$\}$ (i.e., to reduce to
$\ord = 1$), by admissible blowings-up
This completes the proof. (Condition (3) of the theorem has also been satisfied.)
\end{proof}

\begin{remark}\label{rem:high3total}
The proof above provides normal forms analogous to those listed in Remark 
\ref{rem:totaltransf} for the total transform
at every singular point (i.e., nc2 or pp) of the final strict transform.
In order to get the appropriate normal forms also at smooth points of the latter,
we need two more steps (see also \S\ref{subsec:pp}):
\medskip

\noindent
(IV) We apply the desingularization algorithm to the pair given by the final strict 
transform and exceptional divisor, outside $\{$nc2, pp$\}$, until the maximum value of 
$\inv$ is $(1,1,2,0,\infty)$. Note that different components of a centre of the blowings-up involved may meet at the pp locus, but we can separate them as in the proof of 
Theorem \ref{thm:gennc}. We also blow up any closed components of 
$(\inv = (1,1,2,0,\infty))$. We can then clean the latter locus. (See Remark
\ref{rem:totalsingsn=3}.)
\medskip

\noindent
(V) We can now apply Theorem \ref{thm:gennc} outside the closed set given
by $\{$nc2, pp$\}$ together with the locus cleaned up in (IV).
\end{remark}

\begin{remark}\label{rem:onfunctclean}
We have not explicitly considered functoriality in the proof of Theorem \ref{thm:high3},
nor in the proofs of Theorems \ref{thm:min}, \ref{thm:minstrong} (when $n=3$)
in \S\ref{subsec:inv.1}. To ensure functoriality, we have to be a little more careful,
as indicated in Remark \ref{rem:functclean}. For example, in part (I) of the proof of
Theorem \ref{thm:high3} above, we should really blow up until 
$\inv \leq \inv(\text{pp})$, eliminate the components
of $(\inv = \inv(\text{pp}))$ which contain no pp, and then perform the
cleaning blowings-up whether or not the latter is non-empty.
\end{remark}

\subsection{Limits of degenerate pinch points}\label{subsec:dpp}
\begin{remark}\label{rem:limdpp}
Suppose we use the desingularization
algorithm as in the proof of Theorem \ref{thm:high3} above, to blow up until the
maximum value of $\inv$ is $(2,0,3/2,0,2,0,\infty)$. At a point $a$ of a component
of $(\inv = (2,0,3/2,0,2,0,\infty))$ which is generically dpp, $X$ is defined by an
equation of the form
$$
z^2 + u^\al \left(y + u^\be x^2 \right)^2 \left(y - 2u^\be x^2\right) = 0,
$$
where $u^\al$ and $u^\be$ are again monomials in components of the exceptional
divisor.

Using the cleaning lemma as above, we can blow up avoiding dpp to
reduce to the case that $\be = 0$ and $|\al| = 0$ or $1$. If $\al = 0$, then we
have a dpp. 

Suppose that $|\al| = 1$. In this case, the singularity cannot be 
eliminated in the way we handled a similar situation in the proof above.
By a change of variables, we can rewrite
the equation as
$$
z^2 + u y (y + x^2)^2 = 0
$$
(where $u$ here denotes a single variable). 
Blowing up $(z = y = u = 0)$ results in $z^2 + y(u y + x^2) = 0$ --- the 
\emph{exceptional singularity} in Theorem \ref{thm:min}. (See \cite[Sect.\,1]{BLM}.)
\end{remark}

\section{Appendix. Crash course on the desingularization invariant}\label{sec:app}
Our purpose in this section is to provide a working knowledge of the
desingularization invariant, sufficient to understand the way it
is used in our main results without reading all the details of the
desingularization algorithm and the invariant (for example, in \cite{BMinv}, 
\cite{BMfunct}).

Resolution of singularities of a variety $X$ can be described by an \emph{iterative}
algorithm. Desingularization can be realized, according to Hironaka \cite{Hann}
by a sequence of blowings-up. The desingularization invariant $\inv = \inv_X$ 
can be defined iteratively
over a sequence of suitable blowings-up. Resolution of singularities can be realized by choosing, as each successive centre of blowings up, the maximum locus of $\inv$;
this is the approach of \cite{BMinv}.

Every iterative algorithm can be described, in an equivalent way, by a recursive
algorithm. The desingularization algorithm of \cite{BMinv} is presented recursively
in \cite{BMfunct} (as well as in \cite{Ko, W}, for the case of a hypersurface $X$).
The recursive presentation has a certain advantage from the point of view of
formal clarity, but hides the explicit calculations involved in computing the invariant
and its maximal loci, as needed for this article. The brief presentation below
mixes the iterative and recursive aspects. 

We restrict our attention to the case of a hypersurface. Throughout this appendix,
$X \subset Z$ denotes an \emph{embedded hypersurface} defined over
a field $\uk$ of characteristic zero (i.e., $Z$ is a smooth variety and
$X$ is a subvariety of pure codimension $1$, usually reduced).

\subsection{Resolution of singularities}\label{subsec:res}

\begin{theorem}\label{thm:res}
 There is a sequence of blowings-up
\begin{equation}\label{eq:blup}
Z = Z_0 \stackrel{\s_1}{\longleftarrow} Z_1 \longleftarrow \cdots
\stackrel{\s_{t}}{\longleftarrow} Z_{t}\,,
\end{equation}
where each $\s_{j+1}$ has smooth centre $C_j$, such that if
$X_0 = X$, $E_0 = E := 0$ and, for each $j = 0,1,\ldots$,
\begin{enumerate}
\item[(i)] $X_{j+1}$ denotes the strict transform of $X_j$,
\item[(ii)] $E_{j+1}$ denotes the exceptional
divisor of $\s_1\circ\cdots\circ\s_{j+1}$,
\end{enumerate}
then, for each $j$,
\begin{enumerate}
\item $C_j$ and $E_j$ have only simple normal crossings,
\item $C_j \subset \Sing X_j$ or $X_j$ is smooth and
$C_j \subset X_j \cap E_j$,
\item $C_j$ is the maximum locus of an invariant $\inv_X(\cdot)$ (see Remark
\ref{rem:resprecision});
\item $X_t$ is smooth and $X_t,\, E_t$ are snc;
\end{enumerate}
\end{theorem}

Note that (1) implies $E_{j+1}$ is snc. The support of each 
exceptional divisor $E_{j+1}$ has 
ordered components $H^j_1,\ldots,H^{j+1}_{j+1}$
(not necessarily irreducible), where $H^{j+1}_{j+1} := \s_{j+1}^{-1}(C_j)$
and each $H^{j+1}_i$, $i<j+1$ denotes the strict transform in $Z_{j+1}$ of
$H^i_i$. We will denote each $H^j_i$ by $H_i$, for short. 
The ``invariant''
$\inv_X$ is invariant with respect to \'etale (or smooth) morphisms of $Z$ and
ground-field extensions.

\subsection{The desingularization invariant}\label{subsec:inv}
The \emph{desingularization invariant} $\inv = \inv_X$ can defined inductively over any suitable sequence of blowings-up \eqref{eq:blup}.
More precisely, for each $j=0,1,\dots$, we define $\inv$ on $Z_{j+1}$ assuming
that it is defined on $Z_0, \ldots, Z_j$ and each blowing-up $\s_{i+1}$,
$i \leq j$ is \emph{$\inv$-admissible} in the sense that
\begin{enumerate}
\item the centre $C_i \subset Z_i$ of $\s_{i+1}$  is smooth and simple normal 
crossings with
$E_i$, where $E_i$ is the exceptional divisor of $\s_1\circ\cdots\circ\s_i$;
\item $\inv$ is constant on every component of $C_i$.
\end{enumerate}

Write $X_0 := X$. For each $j \geq 0$, let $X_{j+1} \subset Z_{j+1}$
denote the strict transform of $X_j$ by $\s_{j+1}$. If $a \in Z_j$, 
then $\inv_X(a)$ depends on the previous blowings-up.
(A functorial algorithm for resolution of singularities
necessarily has some historical memory; cf. \cite[Ex.\,1.9]{BMmsri},
\cite[\S3.6]{Ko}.) In fact, if $a \in Z_j$, then $\inv_X(a)$ depends only on   
$X_j$ and certain subblocks of the set of components of $E_j$, which
we describe below.

Let $a \in Z_j$. Then $\inv(a)$ has the form
\begin{equation}\label{eq:inv}
\inv(a) = (\nu_1(a), s_1(a), \ldots, \nu_q(a), s_q(a), \nu_{q+1}(a))\,,
\end{equation}
where $\nu_k(a)$ is a positive rational number if $k\leq q$, each $s_k(a)$
is a nonnegative integer, and $\nu_{q+1}(a)$ is either $0$ (the order of an
ideal generated by a unit) or $\infty$ (the order of the zero ideal).
The successive pairs $(\nu_k(a),s_k(a))$ are defined inductively over
\emph{maximal contact} subvarieties of increasing codimension. $\inv(a) = (0)$
if and only if $a \in Z_j\setminus X_j$.

We order finite sequences of the form (\ref{eq:inv}) lexicographically.
Then $\inv(\cdot)$ is upper-semicontinuous on each $Z_j$, and 
\emph{infinitesimally upper-semicontinuous} in the sense that, if $a \in Z_j$,
then $\inv(\cdot) \leq \inv(a)$ on $\s_{j+1}^{-1}(a)$. 

\begin{remark}\label{rem:resprecision}
In Theorem \ref{thm:res}, consider $a \in X_j$ in the maximum locus of $\inv$.
If $\inv(a) = (\ldots,\infty)$, then $C_j =$ maximum locus of $\inv$ is smooth and
$\inv < \inv(a)$ on $\s_{j+1}^{-1}(a)$. If $\inv(a) = (\ldots,0)$, then the maximum locus
of $\inv$ in fact may have
several smooth components --- it is given by the intersection of a
smooth subspace of $Z_j$ with a normal crossings divisor --- and $\inv$ decreases
after finitely many ``monomial'' or ``combinatorial'' blowings-up (centre given
by any component of the maximum locus). See Remark \ref{rem:invs0}.
\end{remark}

We also introduce truncations of $\inv$. Let $\inv_{k+1}(a)$ denote
the truncation of $\inv(a)$ after $s_{k+1}(a)$ (i.e., after the $(k+1)$st 
pair), and let $\inv_{k+1/2}(a)$ denote the truncation of $\inv(a)$ after
$\nu_{k+1}(a)$. ($\inv_{k+1/2}(a) := \inv(a) =: \inv_{k+1}(a)$ if $k \geq q$
in (\ref{eq:inv}). 

Given $a \in Z_j$, let $a_i$ denote the image of $a$ in $Z_i$, $i\leq j$.
(We will speak of \emph{year} $i$ in the history of blowings-up). 
The \emph{year of birth} of 
$\inv_{k+1/2}(a)$ (or $\inv_{k+1}(a)$) denotes the smallest $i$ such that
$\inv_{k+1/2}(a) = \inv_{k+1/2}(a_i)$ (respectively, $\inv_{k+1}(a) =
\inv_{k+1}(a_i)$). 

Let $a \in Z_j$. Let $E(a)$ denote the set of components of $E_j$ which
pass through $a$. The entries $s_k(a)$ of $\inv(a)$ are the sizes of certain 
subblocks of $E(a)$: Let $i$ denote the birth-year of 
$\inv_{1/2}(a) = \nu_1(a)$, and let $E^1(a)$ denote the collection of
elements of $E(a)$ that are strict transforms of components of $E_i$
(i.e., strict transforms of elements of $E(a_i)$). Set $s_1(a) :=
\#E^1(a)$. We define $s_{k+1}(a)$, in general, by induction on $k$: Let $i$ denote
the year of birth of $\inv_{k+1/2}(a)$ and let $E^{k+1}(a)$ denote the 
set of elements of $E(a) \setminus \left(E^1(a) \cup \cdots \cup E^k(a)\right)$
that are strict transforms of components of $E_i$. Set $s_{k+1}(a) :=
\#E^{k+1}(a)$.

Clearly, all $s_k(a) = 0$ in year zero (i.e., if $a \in Z$). We will be
interested in $\inv(a)$ often in the case that all $s_k(a) = 0$, in some
given year $j$.

Given $a \in Z_j$, $\nu_1(a)$ means $\ord_a X_j$.
The entries $\nu_k(a)$ of $\inv(a)$ in general are \emph{residual orders} 
that we define in general in \S\ref{subsec:invgeneral}. We will first 
consider the invariant in year zero,
$$
\inv(a) = (\nu_1(a), 0, \ldots, \nu_q(a), 0, \nu_{q+1}(a))\,,
$$
where the $\nu_k(a)$ are simpler (\S\ref{subsec:inv0}).
In year zero, $\nu_{q+1}(a) = \infty$.
(In general, $\nu_{q+1}(a) = 0$ only if 
$E(a) \setminus \left(E^1(a) \cup \cdots \cup E^q(a)\right) \neq \emptyset$.)

\subsection{Maximal contact}\label{subsec:max}
Let $a \in X$ and let $d := \nu_1(a) = \ord_a X$. Let $f$ denote a local generator 
of $\cI_X$ in a neighhourhood $U$ of $a$ in $Z$
such that $\nu_1(x) \leq d$, $x \in U$. We will write $\cosupp (f,d)$ 
or $\cosupp (\cI_X, d)$ for the locus
of points of order $d$ of $f$ in $U$. Say that a (local) blowing-up
$\s: Z' \to U \subset Z$ with smooth centre $C \subset U$ is \emph{$\ord$-admissible}
if $C \subset \cosupp (f,d)$. 

Let $X'$ denote the strict transform of
$X$ by an ord-admissible blowing-up $\s: Z' \to U \subset Z$ with centre $C$.
At a point of $Z'$, $\cI_{X'}$ is generated by
$f' := y_{\exc}^{-d} f\circ \s$, where $y_{\exc}$ denotes a local generator
of the ideal of the exceptional divisor $\s^{-1}(C)$. We will use the
same notation $X' \subset Z'$ for the strict transform of $X$ by a sequence
of $\ord$-admissible local blowings-up.

A \emph{maximal contact hypersurface} for $\cI_X$ at $a$ denotes a
hypersurface $N = V(z)$, where $z$ is a regular function on a neighbourhood
$U$ as above, such that $\ord_a z =1$, with the property that
$\cosupp(\cI_{X'}, d) \subset N'$ after any sequence of $\ord$-admissible local
blowings-up. ($N' = V(z')$ denotes the strict transform of $N$. See the formal
definition \ref{def:max} below.)

\begin{example}
Suppose that $\cI_X$ has a local generator $f$ which can be written as a Weierstass
polynomial in local coordinates $(x_1,\ldots,x_n)$ at $a=0$,
\begin{equation}\label{eq:wpol}
f(x) = x_n^d + c_{d-1}(\tx) x_n^{d-1} + \cdots c_0(\tx)\,,
\end{equation}
where the coefficients $c_i$ are regular 
(or analytic) functions in $\tx = (x_1,\ldots,x_{n-1})$
such that $\ord_a c_i \geq d-i$. After a coordinate change 
$x_n' = x_n - c_{d-1}(\tx)/d$, we can assume that $c_{d-1} = 0$. We claim that
$z := x_n$ defines a maximal contact hypersurface.

Clearly, $\ord_x f = d$ if and only if $z=0$ and $\ord_{\tx} c_i \geq d-i$,
for all $i$. (Note that $c_i$ can be identified with the restriction to
$N = V(z)$ of the partial derivative $\p_z^if := \p^if/\p z^i$.)

Let $\s: Z' \to U \subset Z$ be an $\ord$-admissible local blowing-up 
with smooth centre $C$. We
can assume that $C = \{x_r = \cdots = x_n =0\}$, after a transformation
of the $\tx$-coordinates. Then $Z'$ can be covered by coordinate charts
$U_{x_j}$, $j=r,\ldots,n$, where the ``$x_j$-coordinate chart'' $U_{x_j}$
has coordinates $(y_1,\ldots,y_n)$ given by
$y_k = x_k/x_j$ if $k = r,\ldots,n$, $k\neq j$, and $y_k = x_k$ otherwise.
The strict transform $X'$ lies in the union of the charts $U_{x_j}$,
$j = r,\ldots,n-1$.

Consider, for example, the chart $U_{x_r}$ with coordinates
$$
y_j = x_j, \,\, j \leq r, \quad y_j = x_j/x_r, \,\, j > r.
$$
In this chart, the strict transform is given by $f'(y) = 0$, where
\begin{align*}
f'(y) &= y_r^{-d}f\circ\s\\
      &= y_n^d + c'_{d-2}(\ty)y_n^{d-2} + \cdots c'_0(\ty)\,,
\end{align*}
and each 
\begin{equation}\label{eq:st}
c'_i(\ty) = y_r^{-(d-i)}c_i\circ \ts\,.
\end{equation}
The strict transform $f'$ has the same form as our original function $f$;
in particular, $\ord_y f' = d$ if and only if $y_n =0$ and 
$\ord_{\ty} c'_i \geq d-i$, for all $i$. Moreover $y_n = z' := y_r^{-1}z\circ\s$;
i.e., $N' = V(y_n)$ is the strict transform of $N$. Our claim follows.
\end{example}

\begin{example}
Suppose that $\cI_X$ has a local generator $f$ of the form $f = z\cdot g$,
where $\ord_a z = 1$. Clearly, in a neighbourhood of
$a$, $\ord_x f = d$ if and only if $z=0$ and $\ord_x g = d-1$. Consider the
transforms $f' := y_\exc^{-d} f\circ\s$, $z' := y_\exc^{-1} z\circ\s$ and
$g' := y_\exc^{-(d-1)} g\circ\s$ by an $\ord$-admissible local blowing-up $\s$.
Then $f' = z'\cdot g'$, and $\ord_y f' = d$ if and only if $z' =0$ and 
$\ord_y g'= d-1$. It follows that $N = V(z)$ is a maximal contact hypersurface.
\end{example}

In general, if $N = V(z)$ is a maximal contact hypersurface for $\cI_X$ at $a$,
then, in a neighbourhood of $a$, $\ord_x f = d$ if and only if $x \in N$ and
$\ord_x \p_z^if|_N \geq d-i$, $i=0,\dots,d-1$ (likewise for the transforms
by an admissible blowing-up). Moreover, the transformation
formula $f' := y_{\exc}^{-d} f\circ \s$ for an $\ord$-admissible blowing-up $\s$
implies the following transformation
rules for the partial derivates $\p_z^if$:
$$
\p_{z'}^i f' = y_{\exc}^{-(d-i)} \p_z^if\circ\s\,,\quad i = 0,\dots,d-1\,.
$$
It therefore makes sense to regard the data given by $(f,d)$ on $Z$
as ``equivalent'' to those given on $N$ by 
$(c_i,d-i) := \left(\p_z^if|_N,d-i\right)$,
$i = 0,\dots,d-1$, with respect to the corresponding transformation rules. 
Since $\dim N = \dim Z - 1$, this idea of equivalence is a basis for induction
on dimension.

Note, however, that we might have $\ord_a c_i > d-i$,
for all $i$. We define
\begin{equation}\label{eq:nu2}
\nu_2(a) := \min_{0\leq i\leq d-1} \frac{\ord_a c_i}{d-i}\,.
\end{equation}
To continue an inductive definition of the invariant, we need to work
not only with data of the form $(f,d)$, where $\ord_a f = d$, but also,
more generally, with a ``marked ideal'' $\ucI = (\cI, d)$, where
$\ord_a \cI \geq d$.

We will return to the invariant in year zero below, but it is convenient
to first formalize the ideas of \emph{marked ideal} and \emph{equivalence}
in a general setting.

\subsection{Marked ideals}\label{subsec:markedideal}
\begin{definitions}\label{def:markedideal}
A \emph{marked ideal} $\ucI$ is a quintuple $\ucI = (Z,N,E,\cI,d)$, 
where: 
\begin{enumerate}
\item
$Z \supset N$ are smooth varieties,
\item
$E = \sum_{i=1}^s H_i$ is a simple normal crossings divisor on $Z$ which is
tranverse to $N$ and \emph{ordered} (the $H_i$ are smooth hypersurfaces
in $Z$, not necessarily irreducible, with ordered index set as indicated),
\item
$\cI \subset \cO_N$ is an ideal,
\item
$d \in \IN$.
\end{enumerate}
The \emph{cosupport} of $\ucI$,
$$
\cosupp \ucI := \{x \in N:\, \ord_x\cI \geq d\}\,.
$$
We say that $\ucI$ is of \emph{maximal
order} if $d = \max\{\ord_x \cI: x \in \cosupp \ucI\}$. The \emph{dimension}
$\dim \ucI$ denotes $\dim N$.
\smallskip

A blowing-up $\s: Z'= \Bl_C Z \to Z$ (with smooth centre $C$) is
\emph{$\ucI$-admissible} (or simply \emph{admissible}) if
$C \subset \cosupp \ucI$, and
$C$, $E$ have only normal crossings.
The \emph{(controlled) transform} of $\ucI$ by an admissible blowing-up
$\s: Z' \to Z$ is the marked ideal $\ucI' = (Z',N',E',\cI',d'=d)$,
where
\begin{enumerate}
\item
$N'$ is the strict transform of $N$ by $\s$,
\item
$E' = \sum_{i=1}^{s+1} H_i'$ (where $H_i'$ denotes the strict transform
of $H_i$, for each $i=1,\ldots,s$, and $H'_{s+1} := \s^{-1}(C)$ --- the
exceptional divisor of $\s$, introduced as the last member of $E'$),
\item
$\cI' := \cI_{\s^{-1}(C)}^{-d}\cdot \s^*(\cI)$ (where $\cI_{\s^{-1}(C)}
\subset \cO_{N'}$ denotes the ideal of $\s^{-1}(C)$).
\end{enumerate}
In this definition, note that $\s^*(\cI)$ 
is divisible by $\cI_{\s^{-1}(C)}^d$ and $E'$ is a normal crossings 
divisor transverse to $N'$, because $\s$ is admissible. We likewise define
the transform by a sequence of admissible blowings-up.
\smallskip

We say that two marked ideals $\ucI$ and $\ucJ$ (with the same ambient
variety $Z$ and the same normal crossings divisor $E$) are \emph{equivalent}
if they have the same sequences of \emph{test transformations} (i.e., every
test sequence for one is a test sequence for the other). \emph{Test
transformations} are transformations of a marked ideal by morphisms of three
possible kinds: admissible blowings-up, projections from products with an
affine line, and \emph{exceptional blowings-up} \cite[Defns.\,2.5]{BMfunct}. 
In particular, if $\ucI$ 
and $\ucJ$ are equivalent, then they have the same cosupport and their
transforms by any sequence of admissible blowings-up have the same cosupport.
The remaining two types of test transformations are used to prove functoriality
properties of the desingularization invariant and algorithm. We refer the
reader to \cite[\S2]{BMfunct} for definitions; 
we do not need these notions explicitly here.
\end{definitions}

In particular, equivalent marked ideals have the same resolution sequences.

\begin{definition}\label{def:markedres}
A \emph{resolution of singularities} of a marked ideal $\ucI = (Z,N,E,\cI,d)$
is a sequence of admissible blowings-up \eqref{eq:blup} after which $\cosupp \ucI'
= \emptyset$. 
\end{definition}

\begin{example}\label{ex:markedres}
Given a hypersurface $X \hookrightarrow Z$ as above, we introduce the marked
ideal
\begin{equation}\label{eq:markedres}
\ucI_X := (Z,Z,\emptyset, \cI_X,1).
\end{equation}
Then a resolution of singularities of $\ucI_X$ (which is functorial with respect
to \'etale morphisms) provides a resolution of singularities of $X$,
\emph{before the last blowing up} for $\ucI_X$. Consider
the resolution sequence for $\ucI_X$. Each centre of blowing-up is smooth and
snc with respect to the exceptional divisor. The last blowing-up leads to
empty cosupport, and the centre of the last blowing-up includes all smooth
points of $X$. It follows that strict 
transform of $X$ coincides with the centre at this step.  So we have resolved the singularities of $X$.
\end{example}

To interpret the data $\{(c_i, d-i)\}$ on $N$ of \S\ref{subsec:max} as a marked
ideal, it it convenient to define sums of marked ideals. In general, we
will shorten the notation $(Z,N,E,\cI,d)$ to $(E,\cI,d)$ or $(\cI,d)$ when
the remaining entries are unambiguous.

\begin{definition}\label{def:sum}
Consider marked ideals
$\ucI = (Z,N,E,\cI,d) = (\cI,d)$ and
$\ucJ = (Z,N,E,\cJ,d) = (\cJ,d)$. 
Define $\ucI + \ucJ$ as $(\cI^{l/d} + \cJ^{l/e}, l)$, where $l = \lcm(d,e)$.
Likewise, for any finite sum.
\end{definition}

It is easy to see: 
\begin{enumerate}
\item
$\cosupp (\ucI + \ucJ)
= \cosupp \ucI \cap \cosupp \ucJ$; 
\item
a blowing-up $\s:Z' \to Z$ is
admissible for $\ucI + \ucJ$ if and only if $\s$ is admissible for both 
$\ucI$ and $\ucJ$, and the transforms satisfy $\ucI' + \ucJ' = (\ucI + \ucJ)'$.
\end{enumerate}

Addition is not associative, but
$\ucI + \ucJ$ is equivalent to $(\ucI^e + \ucJ^d, de)$,
and addition is associative up to equivalence. 

\begin{example} In the notation of \S\ref{subsec:max}, let $\ucJ$ denote
the marked ideal $(Z,Z,\emptyset,\cI_X,d)$, so that $\ucJ|_U = ((f),d)$.
Define the \emph{coefficient marked ideal} $\ucC_U(\ucJ) =
(U,N,\emptyset,\cC,d_{\ucC})$ as the sum of the marked ideals
$((c_i),d-i) = (U,N,\emptyset,(c_i),d-i)$. Then $\ucJ|_U$ is equivalent
to $\ucC_U(\ucJ)$.
\end{example}

\begin{definition}\label{def:invts} {\bf Invariants of a marked ideal.} Given a marked
ideal $\ucI = (Z,N,E,\cI,d)$ and a point $a \in \cosupp \ucI$, we set
\begin{equation}\label{eq:invts}
\mu_a(\ucI) := \frac{\ord_a\cI}{d} \quad \mbox{and} \quad 
\mu_{H,a}(\ucI) := \frac{\ord_{H,a}\cI}{d},\,\, H \in E.
\end{equation}
($\ord_{H,a}\cI$ denotes the \emph{order} of $\cI \subset \cO_N$
\emph{along} $H|_N$ at $a$;\, i.e., the largest $\mu \in \IN$ such that
$\cI_a \subset \cI_{H|_N,a}^\mu$.)
\end{definition}

Both $\mu_a(\ucI)$ and $\mu_{H,a}(\ucI)$ depend only on the equivalence 
class of $\ucI$ and $\dim N$ \cite[\S6]{BMfunct}.
 
\begin{definition}\label{def:max} {\bf Maximal contact.}
Let $\ucI = (Z,N,E,\cI,d) = (\cI,d)$ be a marked ideal
and let $a \in N$. Let $z$ denote a 
regular function on a neighbourhood of $a$ in $N$ such that $\ord_a =1$. Then
$P := V(z)$ is a \emph{maximal contact hypersurface} for $\ucI$ at $a$ if
$P$ is transverse to $E$ and
$(\cI,d) + ((z), 1)$ is equivalent to $(\cI,d)$ on a neighbourhood of $a$ in $Z$.
\end{definition}

\begin{lemma}\label{lem:max}
A marked ideal $\ucI = (Z,N,\emptyset,\cI,d)$ admits a maximal contact hypersurface
at $a \in N$ if and only if $\ord_a \cI = d$ (i.e., $\ucI$ is of maximal order on a
sufficiently small neighbourhood of $a$).
\end{lemma} 

\begin{proof}
The ``only if'' direction is consequence of invariance of $\mu_a(\ucI)$.
In the other direction, if $\ord_a \cI = d$, then there is a local section $f$
of $\cI$ at $a$ and a partial derivative
$\p^\al := \p^\al / \p x^\al$ of order $d-1$, with respect to local coordinates
of $N$, such that $z:= \p^\al f$ has order $1$ at $a$. Then $P:= V(z) \subset N$
is a maximal contact hypersurface at $a$ \cite[\S4]{BMfunct}.
\end{proof}

\subsection{Coefficient ideals} \label{subsec:coeff}
We now formalize the coefficient data
$\{(c_i,d-i)\}$ of \S\ref{subsec:max} as a marked ideal.
Let $\ucI = (Z,N,E,\cI,d) = (\cI,d)$ be a marked ideal of maximal order, and
let $a \in \cosupp \ucI$. Suppose that $P = V(z)$ is a maximal contact
hypersurface for $\ucI$, in some neighbourhood $U$ of $a$. In a suitable $U$,
we can find a system of local coordinates
$(x_1,\ldots,x_n)$ for $N$ such that $x_n = z$ and the
components of $E$ are given by $x_i = 0$, $i = 1,\dots,r<n$. 
Let $\cD_z(\cI)$ denote the ideal generated by $f$, $\p f/\p z$,
for all $f \in \cI$, and let $\ucD_z(\ucI)$ denote the marked ideal
$\left(\cD_z(\cI),d-1\right)$. 
For $j\geq 2$, we inductively set
$\cD_z^j(\cI) := \cD_z\left(\cD_z(\cI)\right)$, and we
define marked ideals
\begin{align*}
\ucD_z^j(\ucI) &:= \left(\cD_z^j(\cI),d-j\right),\quad j=0,\dots,d-1,\\
\ucC_z^{d-1}(\ucI) &:= \sum_{j=0}^{d-1}\ucD_z^j(\ucI)\\
                   &= \left(\cC_z^{d-1}(\ucI),d_{\ucC}\right),\quad\text{say}. 
\end{align*}		    
We define the \emph{coefficient (marked) ideal} $\ucC_z(\ucI)$ as
the restriction of the latter to the maximal contact hypersurface $P$; i.e.,
$$
\ucC_z(\ucI) := \left(U,P,E, \cC_z^{d-1}(\ucI)|_P, d_{\ucC}\right).
$$
Then $\ucC_z(\ucI)$ is equivalent to $\ucI$ (in the chart $U$),
essentially by the calculations in \S\ref{subsec:max} (see \cite[\S4]{BMfunct}). 

\begin{example}\label{ex:split} Suppose that $E = \emptyset$ and 
$\cI$ is a principal ideal generated by $f(x)$ as in \eqref{eq:wpol}. Assume
that $c_{d-1} = 0$. Set $z:= x_n$. Then $P = V(z)$  is a maximal contact
hypersurface, and the coefficient
ideal $\ucC_z(\ucI) = \sum_{i=0}^{d-2}((c_i), d-i)$.

Suppose that $f(x)$ splits; i.e.,
$$
z^d + c_{d-2}(\tx) z^{d-2} + \cdots c_0(\tx) = (z-b_1(\tx))\cdots(z-b_d(\tx)).
$$
Then $\ucC_z(\ucI)$ is equivalent to the marked ideal $\sum_{j=1}^d((b_j), 1)$.
This follows from the fact the $\ord_a b_j \geq k$, for all $j$, if and only if
$\ord_a \s_i \geq ki$, for all $i$, where $\s_i$ denotes the $i$th elementary
symmetric function of the $b_j$.
\end{example}

\begin{remark}\label{rem:issues}
In general, since the coefficient ideal $\ucC_z(\ucI)$ is equivalent to $\ucI$
(in a chart $U$ as above), any resolution of singularities of $\ucC_z(\ucI)$
is a resolution of singularities of $\ucI$ over $U$ (as in Definitions \ref{def:markedideal}).
Since $\dim \ucC_z(\ucI) < \dim \ucI$, the idea is to use the coefficient ideal as 
a basis for induction on dimension. There are two main problems involved in carrying
out this idea.
\smallskip

(1) Passage from $\ucI$ to $\ucC_z(\ucI)$ requires that $\ucI$ be of maximal
order, so that it admits a maximal contact hypersurface (according to 
Lemma \ref{lem:max}). But $\ucC_z(\ucI)$ is not necessarily
of maximal order, so we cannot \emph{a priori}
repeat the construction inductively.

Morevover, maximal contact is not unique. Local centres of blowing up chosen by an
inductive construction as above need not \emph{a priori} glue together to give a global
centre of blowing up. This gluing problem can be resolved by iterating a suitable inductive
construction in decreasing dimension to define a desingularization invariant (or, as in 
\cite{BMfunct}, by using functoriality properties of equivalence classes of marked ideals
to make a stronger inductive assumption that guarantees gluing).
\smallskip

(2) In general, a marked ideal $\ucI = (Z,N,E,\cI,d)$ of maximal order 
admits a maximal contact hypersurface 
$P = V(z)$, according to Lemma \ref{lem:max}, only provided
that $E=\emptyset$ (for example, in year zero).
\end{remark}

Item (1) of the Remark is treated using the constructions in 
\S\S\ref{subsec:monres}, \ref{subsec:compan} and item (2) using \S\ref{subsec:bdry}. 

\subsection{Monomial and residual ideals}\label{subsec:monres}
In general, given a marked ideal $\ucI = (Z,N,E,\cI,d) = (\cI, d)$, we can
factor $\cI$ as
\begin{equation*}
\cI = \cM(\ucI)\cdot\cR(\ucI),
\end{equation*}
where $\cM(\ucI)$ is a product of the ideals $\cI_H$ of the components $H$ of $E$,
and $\cR(\ucI)$ is divisible by no such exceptional divisor. We call $\cM(\ucI)$ the
\emph{monomial} or \emph{divisorial part} and $\cR(\ucI)$ the \emph{residual}
or \emph{nonmonomial part} of $\cI$.

We define the \emph{residual multiplicity} of $\ucI$ at a point $a \in \cosupp \ucI$,
\begin{equation*}
\nu_{\ucI}(a) := \frac{\ord_a \cR(\ucI)}{d}.
\end{equation*}
Then
\begin{equation*}
\nu_{\ucI}(a) = \mu_a(\ucI) - \sum_{H \in E} \mu_{H,a}(\ucI)
\end{equation*}
(cf. Definition \ref{def:invts}), so that $\nu_{\ucI}(a)$ depends only on the
equivalence class of $\ucI$. 

We use the residual multiplicity to define the term $\nu_2(a)$ in $\inv$,
and inductively to define $\nu_j(a)$, $j\geq 2$. (See \S\ref{subsec:inv0} and Definition
\ref{def:invgeneral}.)

Let $\ord\, \cR(\ucI)$ denote the maximum order of $\cR(\ucI)$
on $\cosupp \ucI$. Then the \emph{residual (marked) ideal}
\begin{equation*}
\ucR(\ucI) := (\cR(\ucI), \ord\, \cR(\ucI)) = (Z,N,E, \cR(\ucI), \ord\, \cR(\ucI))
\end{equation*}
is a marked ideal of maximal order.
\smallskip

In general, a blowing-up that is admissible for $\ucR(\ucI)$ need not be admissible
for $\ucI$. If $\cM(\ucI) = 1$, however (for example, in year zero), 
then $\ucR(\ucI) = \ucI$ and any blowing-up
that is $\ucR(\ucI)$-admissible will also be $\ucI$-admissible. This is enough to
define the invariant in year zero. 

\begin{remark}\label{rem:localcalc}
In order to calculate the resolution invariant at a point $a$ in any year of the
resolution history, we make the construction above locally at $a$. In particular,
we can identify $E$ with the set $E(a)$ of components of $E$ at $a$, and
$\ord\, \cR(\ucI) = \ord_a \cR(\ucI)$. This localization of the construction will be
assumed in the computation below.
\end{remark}

\subsection{The invariant in year zero}\label{subsec:inv0}
All $s_i = 0$.
Let $\ucI^0 := \ucI_X$ (see Example \ref{ex:markedres}). Then $\cR(\ucI^0) = \cI^0$. Consider
$a \in \cosupp \ucI^0$. Then $\nu_{\ucI^0}(a) = \ord_a \cI_X = \nu_1(a)$. We set:
\begin{itemize}
\item[{}]
$\ucJ^0 := \ucR(\ucI^0)$. Then $\ucJ^0$ is of maximal order. Let $P = V(z)$ be a
maximal contact hypersurface for $\ucJ^0$ at $a$.
\smallskip

\item[{}]
$\ucI^1 :=$ the coefficient ideal 
$\ucC_z(\ucJ^0) = \left(Z,P,\emptyset,\cC(\ucJ^0), d_{\ucC}\right)$. 
\end{itemize}
\smallskip

We define 
$$
\nu_2(a) := \nu_{\ucI_1}(a) = \ord_a \cR(\ucC_z(\ucJ^0)) / d_{\ucC}
$$ 
(of course,
here in year zero, $\cR(\ucC_z(\ucJ^0)) = \cC(\ucJ^0)$), and
we iterate the preceding construction: Set $\ucJ^1 := \ucR(\ucI^1)$. Then $\ucJ^1$ is of maximal order, so it admits a maximal
contact hypersurface $Q = V(w)$ in $P$; $Q$ is of the form $V(z,w)$ in a coordinate chart
of $Z$ --- a ``codimension two maximal contact subspace'', etc. We thus define 
$\nu_3(a)$,\,$\ldots$\,. At a certain step, the coefficient ideal
$\ucI^q = \ucC_\cdot(\ucJ^{q-1})$ becomes zero (e.g., we might run out of variables). Then we put $\nu_{q+1}(a) 
:= \infty$ and $\inv(a)= (\nu_1(a), 0, \nu_2(a), 0, \ldots, 0, \nu_{q+1}(a))$. The locus
of points $(\inv = \inv(a))$ (the locus of points where $\inv = \inv(a)$) is (locally) the last maximal contact subspace,
of codimension $q$.

\begin{example}\label{ex:ppyear0}
Let $X$ denote the hypersurface $(z^2 + xy^2 = 0)$ in $Z = \IA^3$. We show
that (in year zero), $\inv(0) = (2,0,3/2,0,1,0,\infty)$ and $(\inv = \inv(0))$ is $C_0 = \{0\}$; this will be the first centre of blowing-up 
in the resolution algorithm. The calculations needed to compute $\inv(0)$
according to the preceding definition are presented in the following table.
The marked ideal $\ucI^{i+1}$ in each row $i+1$ of the table lives on the maximal
contact subspace (of codimension $i+1$)
in row $i$. Each $\ucI^{i+1}$ is the coefficient ideal of $\ucJ^i$.
It is clear that $(\inv = \inv(0))$ is the last maximal contact subspace $(z=y=x=0)$.
\medskip

\renewcommand{\arraystretch}{1.5}
\begin{tabular}{c | c | c | c }
codim $i$ & marked ideal $\ucI^i$ & residual ideal $\ucJ^i$ & maximal contact \\\hline
0 & $(z^2 + xy^2,1)$ & $(z^2 + xy^2,2)$ & $(z=0)$\\\hline
1 & $(xy^2,2)$          & $(xy^2, 3)$          & $(z=y=0)$\\\hline
2 & $(x,1)$                & $(x,1)$                & $(z=y=x=0)$\\\hline
3 & $0$ &&\\
\end{tabular}
\end{example}
\smallskip

\subsection{Companion ideals}\label{subsec:compan}
We use the notation of \S\ref{subsec:monres}. Recall that, in general, a
blowing-up that is admissible for $\ucR(\ucI)$ need not be admissible for $\ucI$.
We define the \emph{companion
ideal} $\ucG(\ucI)$ as
$$
\ucG(\ucI) := 
\begin{cases}
(\cR(\ucI), \ord\,\cR(\ucI)) + (\cM(\ucI), d - \ord\,\cR(\ucI)), & \ord\,\cR(\ucI)) < d \\
(\cR(\ucI), \ord\,\cR(\ucI)), & \ord\,\cR(\ucI)) \geq d 
\end{cases}
$$

It is not difficult to see that $\cosupp \ucG(\ucI) = \cosupp \ucR(\ucI) \cap 
\cosupp \ucI$ and thus that $\ucG(\ucI)$-admissible blowings-up are
also $\ucI$-admissible.
Moreover, the equivalence class of $\ucG(\ucI)$ depends
only on the equivalence class of $\ucI$; this is a consequence of the same
property for the invariants \eqref{eq:invts} (see \cite[Cor.\,5.3]{BMfunct}).

This is enough to define the invariant
at a point $a$ \emph{in any year of the blowing-up history}, provided that all $s_i(a) = 0$.
We simply use the year zero definition of \S\ref{subsec:inv0} with one change:
For each $i$, we 
take $\ucJ_i := \ucG(\ucI_i)$.

\begin{remark}\label{rem:invs0}
In the preceding definition, note that each $\nu_{i+1}(a) := \nu_{\ucI_i}(a)$, where
the latter is still the residual multiplicity as defined in \S\ref{subsec:inv0}. But
the change in the definition of the $\ucJ^i$ may result in a change in $\nu_i(a)$,
$i \geq 2$, and it might result in a change in the last term $\nu_{q+1}(a)$ of $\inv(a)$:

In the current situation, we will arrive at a certain step $q$ where either
$\cI^q = 0$ or $\cI^q = \cM(\ucI^q)$. In the former case, we put $\nu_{q+1}(a) 
:= \infty$, as in \S\ref{subsec:inv0}. In the latter case, we put $\nu_{q+1}(a) 
:= 0$ (the order of $\cR(\ucI^q)$). This is the \emph{monomial case} of
resolution of singularities; see \cite[Sect.\,5, Step\,II, Case\,A]{BMfunct}.
We do not need the invariant in the case that $\nu_{q+1}(a) = 0$ in this article,
but monomial resolution intervenes in the cleaning lemma (Section 2).
\end{remark}

\subsection{Coefficient ideals with boundary}\label{subsec:bdry}
The construction of this subsection is needed to treat the terms $s_i(a)$, in general.
Let $\ucJ = (Z,N,E,\cJ,d)$ denote a marked ideal of maximal order. We call $E$
the \emph{boundary} of $\ucJ$.
Set $\ucJ_\emptyset := (Z,N,\emptyset,\cJ,d)$. Then, locally, 
$\ucJ_\emptyset$ admits a maximal contact
hypersurface $P = V(z)$, by Lemma \ref{lem:max}. However, $P$ need not be
snc with respect to $E$. 

We ``add the boundary to the coefficient ideal'' (see \eqref{eq:bdry}) 
to ensure that the centre of 
blowing up will lie in all components of the boundary, so will automatically be
snc with respect to the boundary divisor.

At any point $a$ of $\cosupp \ucJ$, the boundary determines a marked ideal
$\sum(\cI_H,1)$, where the sum is over all components $H$ of $E$ such that
$a \in H$. At $a$, the coefficient ideal plus boundary is given by
\begin{equation}\label{eq:bdry}
\ucI' := \ucC_z(\ucJ_\emptyset) + \sum(\cI_H|_{(z=0)},1)
\end{equation}

Note that $\ucI'$ itself has empty boundary. Resolution of singularities of $\ucI'$
involves centres in the maximal contact hypersurface $(z=0)$ and its 
successive strict transforms. During the resolution process, the \emph{new}
exceptional divisors that accumulate are automatically transverse to (the strict
transform of) $(z=0)$, and the \emph{old} exceptional divisors (the boundary
above) will be moved away.

\begin{remark}\label{rem:iterative} 
Given a marked ideal $\ucI = (Z,N,E,\cI,d)$, set $E(\ucI) := E$.

Again consider $\ucI = \ucI_X$. Let $a$ denote a point in year zero. Write $E^1(a)
= E(a)$.
Resolution of singularities of the companion ideal $\ucJ = \ucG(\ucI)$ at $a$
provides a sequence of admissible blowings-up for $\ucI$ over $a$. Consider
a point $b$ over $a$,
in any year of the resolution history for $\ucJ$. 

Suppose that $b \in \cosupp \ucJ$. Then 
$\nu_1(b) = \nu_1(a)$. Let $E^1(b)$ denote the set of transforms of elements of $E(a)$
at $b$ (the ``old exceptional divisors''). Note also that $\ucC_z(\ucJ_\emptyset)$ has
accumulated a set of ``new exceptional divisors'' $E(b)\setminus E^1(b)$ at $b$.
Moreover, $\ucC_z(\ucJ_\emptyset)$ has a maximal contact hypersurface at $b$,
transformed from year zero, so transverse to the new exceptional divisors.

On the other hand, suppose that $b \notin \cosupp \ucJ$. Then 
$\nu_1(b) < \nu_1(a)$. When the order first drops (the ``year of birth'' of $\inv_{1/2}
= \inv_{1/2}(b)$), we choose a new companion
ideal $\ucJ$ and a new maximal contact hypersurface for $\ucJ_\emptyset$ at $b$,
which need not be transverse to $E(b)$. Then we set $E^1(b) := E(b)$ and
repeat the process above.

Then, at a point $c$ in any year of the resolution history for $\ucI$, the boundary in
\eqref{eq:bdry} is $E^1(c)$ and the coefficient ideal plus boundary is
\begin{equation}\label{eq:bdrygeneral}
\ucC_z(\ucJ_\emptyset) + \sum_{H \in E^1(c)}(\cI_H|_{(z=0)},1),
\end{equation}
with $E(\ucC_z(\ucJ_\emptyset)) = E(c)\setminus E^1(c)$.
\end{remark}

%
In the iterative construction of the invariant, the boundary phenomenon
occurs on maximal contact subspaces of every codimension $i$. The
boundary components added to the coefficient ideal on a maximal contact
subspace of codimension $i$ at a point $a$ are the elements of $E^i(a)$;
i.e, the components of the exceptional divisor counted by $s_i(a)$ (see
\S\ref{subsec:inv} and Remark \ref{rem:pract}). 

\subsection{The desingularization invariant and an example computation}
\label{subsec:invgeneral}
We begin with a definition of $\inv$, in the general case.

\begin{definition}\label{def:invgeneral} {\bf The desingularization invariant.} 
We repeat the iterative scheme in \S\S\ref{subsec:inv0}
and \ref{subsec:compan} above, with the changes need to accommodate the boundary
terms.

As in \S\ref{subsec:inv}, we assume, by induction, that $\inv$ has been
defined up to year $j$ (so that blowings-up have been determined, up to
$\s_{j+1}: Z_{j+1} \to Z_j$. Let $\ucI^0$ denote the transform in year $j+1$ of $\ucI_X$ (see Example \ref{ex:markedres}). Consider
$a \in \cosupp \ucI^0$. Then $\nu_{\ucI^0}(a) = \ord_a \cI_X = \nu_1(a)$. We
define $E^1(a)$ as in \S\ref{subsec:inv} or \S\ref{subsec:bdry}
and set $s_1(a) = \#E^1(a)$. We take:
\begin{itemize}
\item[{}]
$\ucJ^0 := \ucG(\ucI^0)$. Then $\ucJ^0$ is of maximal order. Let $P = V(z)$ be
a maximal contact hypersurface for $\ucJ^0$ at $a$ (see Remark \ref{rem:iterative}).
\smallskip

\item[{}]
$\ucI^1 :=$ the coefficient ideal plus
boundary, i.e.,
$$
\ucI^1 := \ucC_z(\ucJ^0_\emptyset) + \sum_{H \in E^1(a)}(\cI_H|_{(z=0)}, 1),
$$
as in \eqref{eq:bdrygeneral}.
\end{itemize}

We define 
$$
\nu_2(a) := \nu_{\ucI_1}(a), \quad s_2(a) = \#E^2(a),
$$
with $E^2(a)$ as in \S\ref{subsec:inv}, and iterate the construction. 

We finish when $\nu_{q+1}(a) = 0$ or $\infty$, as in Remark \ref{rem:invs0}.
\end{definition}

\begin{remark}\label{rem:fin}
If $\nu_{q+1}(a) = \infty$, then the locus $\inv = \inv(a)$ is the maximal
contact subspace of codimension $q$. If, in addition, all $s_i(a) = 0$, then
the latter is transverse to the exceptional divisor.
\end{remark}

\begin{remark}\label{rem:pract}
In practical terms, $\ucI^i$ lives on a maximal contact subspace of codimension $i$.
To pass from $\ucI^i$ to the companion ideal $\ucJ^i$, we use the factorization
$\cI^i = \cM(\ucI^i)\cR(\ucI^i)$ of \S\ref{subsec:monres}. At a point $a$,
$\cM(\ucI^i)$ is a monomial in the exceptional divisors in $E(a)\setminus (E^1(a) \cup
\cdots \cup E^i(a))$, which are transverse to $N^i$ (the ``new'' exceptional divisors
in codimension $i$). The ``old'' exceptional divisors in $E^i(a)$ are transformed from
the year of birth of $\inv_{i-1/2} = \inv_{i-1/2}(a)$. They are counted by $s_i(a)$
rather than considered elements of $E(\ucI^i)$.
\end{remark}

\begin{example}\label{ex:ppres}
We compute the blowings-up given by the desingularization 
algorithm for the pinch-point singularity, after the first blowing-up given in
Example \ref{ex:ppyear0}. The following table provides the computations needed
to find the invariant and the centre $C$ of the blowing-up at the origins of
the charts corresponding to the coordinate substitutions indicated. Note that
the pinch-point singularity persists to year two. The strict transform of the
pinch-point hypersurface in the year-one chart exhibited lies in the union of 
the two year-two charts shown. The calculations at a given point 
provide the next
centre of blowing up over a neighbourhood of that point; globally, the maximum
locus of the invariant will be blown up first.

In each
subtable, the passage from $\ucJ^i$ to $\ucI^{i+1}$ is given by taking the
coefficient ideal plus boundary, on the maximal contact subspace of codimension $i+1$.
\bigskip

\renewcommand{\arraystretch}{1.5}
\begin{tabular}{c | c | c | c | c}
\hline
codim & marked ideal & companion ideal  & maximal & boundary\\[-.25cm]
$i$ & $\ucI^i$ & $\ucJ^i = \ucG(\ucI^i)$ & contact & $E^i$ \\\hline
\multicolumn{5}{ c }{}\\[-.5cm]
\multicolumn{5}{ l }{{\bf Year one.}\,  Coordinate chart $(x,xy,xz)$}\\[.1cm]\hline
0 & $(x(z^2 + xy^2),1)$ & $(z^2 + xy^2,2)$ & $(z=0)$&\\\hline
1 & $(xy^2,2)$          & $(y^2, 2)$          & $(z=y=0)$& (x=0)\\\hline
2 & $(x,1)$                & $(x,1)$                & $(z=y=x=0)$&\\\hline
3 & $0$ &&&\\\hline
\multicolumn{5}{ l }{$\inv(0) = (2,0,1,1,1,0,\infty)$,\,  $C_1 = \{0\}$}\\
\multicolumn{5}{ c }{}\\[-.4cm]
\multicolumn{5}{ l }{{\bf Year two.}\,  Coordinate chart $(x,xy,xz)$}\\[.1cm]\hline
0 & $(x^2(z^2 + xy^2),1)$ & $(z^2 + xy^2,2)$ & $(z=0)$&\\\hline
1 & $(xy^2,2)$          & $(y^2, 2)$          & $(z=y=0)$& \\\hline
2 & 0              &     & &\\\hline
\multicolumn{5}{ l }{$\inv(0) = (2,0,1,0,\infty)$,\,  $C_2 = (z=y=0)$}\\
\multicolumn{5}{ c }{}\\[-.4cm]
\multicolumn{5}{ l }{{\bf Year two.}\,  Coordinate chart $(xy,y,yz)$}\\[.1cm]\hline
0 & $(xy(z^2 + xy),1)$ & $(z^2 + xy,2)$ & $(z=0)$&\\\hline
1 & $(xy,2)$          &          & & \\\hline
\multicolumn{5}{ l }{$\inv(0) = (2,0,0)$,\,  $C_2 = \{0\}$}\\
\end{tabular}
\smallskip

\end{example}

\bibliographystyle{amsplain}

\end{document}